\documentclass[11pt,twoside,a4paper]{amsart}

\usepackage[utf8]{inputenc}

\usepackage{enumerate}

\usepackage[colorlinks,citecolor=blue,urlcolor=blue]{hyperref}

\usepackage[english]{babel}

\newtheorem{proposition}{Proposition}
\newtheorem{theorem}{Theorem}
\newtheorem{lemma}{Lemma}
\newtheorem{corollary}{Corollary}

\theoremstyle{definition}

\theoremstyle{remark}
\newtheorem{remark}{Remark}
\newtheorem{example}{Example}

\DeclareMathOperator{\Dom}{Dom}

\newcommand{\NN}{\mathbb{N}}

\newcommand{\RR}{\mathbb{R}}
\newcommand{\ZZ}{\mathbb{Z}}
\newcommand{\CC}{\mathbb{C}}

\newcommand{\calA}{\mathcal{A}}
\newcommand{\calV}{\mathcal{V}}
\newcommand{\calB}{\mathcal{B}}

\newcommand{\calF}{\mathcal{F}}
\newcommand{\calH}{\mathcal{H}}
\newcommand{\calX}{\mathcal{X}}

\newcommand{\Id}{\mathrm{Id}}

\renewcommand{\Re}{\textrm{Re}}

\newcommand{\sprod}[2]{\langle {#1}, {#2} \rangle}
\newcommand{\norm}[1]{\lVert {#1} \rVert}
\newcommand{\abs}[1]{\lvert {#1} \rvert}
\newcommand{\sym}[1]{\mathrm{Re} \left[ {#1} \right] }
\newcommand{\Imag}[1]{\mathrm{Im} \left[ {#1} \right]}

\newcommand{\sigmaP}[1]{\sigma_{\mathrm{p}}(#1)}

\newcommand{\lrparen}[1]{
  \left(\mkern-6mu #1 \mkern-6mu\right)}

\author{Grzegorz Świderski}
\email{gswider@math.uni.wroc.pl}
\address{
	Instytut Matematyczny\\
	Uniwersytet Wrocławski\\
	Pl. Grunwaldzki 2/4\\
	50-384 Wrocław\\
	Poland}

\title{Spectral properties of block Jacobi matrices}



\begin{document}
\selectlanguage{english}

\begin{abstract}
	We study the spectral properties of bounded and unbounded Jacobi matrices whose entries
	are bounded operators on a complex Hilbert space. In particular, we formulate conditions assuring 
	that the spectrum of the studied operators is continuous. Uniform asymptotics of generalised eigenvectors 
	and conditions implying complete indeterminacy are also provided.
\end{abstract}

\maketitle
   
\section{Introduction}
   	Let $\calH$ be a~complex Hilbert space. Consider two sequences $a = (a_n \colon n \geq 0)$ 
   	and $b = (b_n \colon n \geq 0)$ of bounded linear operators on $\calH$ such that for every 
   	$n \geq 0$ the operator $a_n$ has a~bounded inverse and $b_n$ is self-adjoint.
   	Then one defines the symmetric tridiagonal matrix by the formula\footnote{By $X^*$ we denote the 
   	adjoint operator to $X$.}
   	\begin{equation*}
    	  \calA =
    	  \begin{pmatrix}
    	     b_0 & a_0 & 0   & 0      &\ldots \\
    	     a_0^* & b_1 & a_1 & 0       & \ldots \\
    	     0   & a_1^* & b_2 & a_2     & \ldots \\
    	     0   & 0   & a_2^* & b_3   &  \\
    	     \vdots & \vdots & \vdots  &  & \ddots
    	  \end{pmatrix}.
   	\end{equation*}
   	The action of $\calA$ on \emph{any} sequence of elements from $\calH$ is defined by the formal matrix
   	multiplication. Let the operator $A$ be the minimal operator associated with $\calA$. 
   	Specifically, by $A$ we  mean the closure in $\ell^2(\NN; \calH)$ of the restriction of $\calA$ to the set of 
   	the sequences of finite support. Let us recall that
   	\[
    	  \sprod{x}{y}_{\ell^2(\NN; \calH)} = \sum_{n=0}^\infty \sprod{x_n}{y_n}_\calH, \quad \ell^2(\NN; \calH) = 
    	  \{ x \in \calH^\mathbb{N} \colon \sprod{x}{x}_{\ell^2(\NN; \calH)} < \infty \}.
   	\]
   	The operator $A$ is called a~\emph{block Jacobi matrix}. It is self-adjoint provided the Carleman condition 
   	is satisfied, i.e.
   	\begin{equation} \label{eq:27}
   		\sum_{n=0}^\infty \frac{1}{\norm{a_n}} = \infty,
   	\end{equation}
   	where $\norm{\cdot}$ is the operator norm (see \cite[Theorem VII-2.9]{Berezanski1968}).
   
   	Block Jacobi matrices are related to such topics as: matrix orthogonal polynomials (see \cite{Damanik2008}),
   	the matrix moment problem (see \cite{Duran2001a}), difference equations of finite order 
   	(see \cite{Duran1995}), partial difference equations (see \cite{Berezanski1968}), 
   	level dependent quasi-birth–death processes (see \cite{Dette2007} and references therein). 
   	For further applications we refer to \cite{Koelink2016, Sinap1996}.
	
	The theory of block Jacobi matrices is much less developed than the scalar ones, i.e. 
	corresponding to $\calH = \CC$. The aim of this paper is to provide extensions of results
	obtained in \cite{Swiderski2016, Swiderski2017} for $\calH = \RR$ to the case of arbitrary $\calH$.
	It is of interest as we provide new results even for $\calH = \CC^d$ with $d \geq 1$, i.e.
	the most common (apart from $\RR$) studied case.

	Originally, we were interested in the \emph{unbounded} case, i.e.
	\[
		\lim_{n \rightarrow \infty} \norm{a_n^{-1}} = 0.
	\]
	But it seems that even the bounded case is not well understood (see \cite{JanasNaboko2015, Sahbani2016}).
	Therefore, we present a~unified treatment of both bounded and unbounded cases. In the unbounded
	case the formulation of our results is simpler.

	In the proofs of the presenting theorems we will use the following notion. 
	A non-zero sequence $(u_n : n \geq 0)$ will be called a~\emph{generalised eigenvector} associated
	with $z \in \CC$ if it satisfies the recurrence relation
	\[
		a_{n-1}^* u_{n-1} + b_n u_n + a_n u_{n+1} = z u_n, \quad (n \geq 1).
	\]
	In Section~\ref{sec:genEig} we show the correspondence between asymptotic behaviour of generalised 
	eigenvectors and the spectral properties of $A$.

	The first main result of this article is Theorem~\ref{thm:A}, which generalises the results obtained in
	\cite{Swiderski2016} to the operator case. Its formulation involves an additional parameter sequence 
	$\alpha = (\alpha_n : n \geq 0)$. In Section~\ref{sec:specialCases} we present some of the possible choices 
	of $\alpha$. The following Theorem is a~special case of Theorem~\ref{thm:A} (obtained for $\alpha_n = a_n$).
	\begin{theorem} \label{thm:spec:2}
	Assume
	\[
		\lim_{n \rightarrow \infty} \norm{a_n^{-1}} = 0, \quad \lim_{n \rightarrow \infty} \norm{a_n^{-1} b_n} = 0
	\]
	and\footnote{For a~self-adjoint operator $X \in \calB(\calH)$ we define $X^-$ by the spectral theorem.}
	\begin{enumerate}[(a)]
		\item $\begin{aligned}[b]
			\sum_{n=1}^\infty \frac{\norm{[a_{n+1} a_{n+1}^* - a_n^* a_n]^-}}{\norm{a_n}^2} < \infty,
		\end{aligned}$ \label{thm:spec:2:eq:1}
		
		\item $\begin{aligned}[b]
			\sum_{n=1}^\infty \frac{\norm{a_n b_{n+1} - b_n a_n}}{\norm{a_n}^2} < \infty.
		\end{aligned}$ \label{thm:spec:2:eq:2}
		
		\item $\begin{aligned}[b]
			\sum_{n=0}^\infty \frac{1}{\norm{a_n}^2} = \infty.
		\end{aligned}$ \label{thm:spec:2:eq:3}
	\end{enumerate}
	Then the operator $A$ is self-adjoint. Moreover\footnote{By $\sigma(A)$ we denote spectrum of the operator 
	$A$, whereas $\sigmaP{A}$ is the set of its eigenvalues.}, $\sigma(A) = \RR$ and 
	$\sigmaP{A} = \emptyset$ provided
	\[
		\lim_{n \rightarrow \infty} \left\| \frac{a_n}{\norm{a_n}} - C \right\| = 0,
	\]
	where $C$ is invertible.
	\end{theorem}

	Before we formulate the next result we need a~definition. Given a~positive integer $N$, we define the 
	total $N$-variation $\calV_N$ of a~sequence of vectors $x = \big(x_n : n \geq 0 \big)$ from a~vector 
	space~$V$ by
	\[
		\calV_N(x) = \sum_{n = 0}^\infty \norm{x_{n + N} - x_n}.
	\]
	Observe that if $(x_n : n \geq 0)$ has a finite total $N$-variation then for each 
	$j \in \{0, \ldots, N-1\}$ a~subsequence $(x_{k N + j} : k \geq 0)$ is a~Cauchy sequence. 

   	The following Theorem is interesting even for $N=1$. Since recently block periodic Jacobi matrices have
   	obtained some attention (see \cite{Damanik2015, JanasNaboko2015}) we formulate it for an arbitrary 
   	natural number $N$.
   	\begin{theorem} \label{thm:B}
   		Let $N \geq 1$ be an integer. Assume
   		\[
   			\calV_N (a_n^{-1} : n \geq 0) + \calV_N (a_n^{-1} b_n : n \geq 0) + 
	   		\calV_N (a_n^{-1} a_{n-1}^* : n \geq 1) < \infty.
   		\]
   		Let
   		\begin{enumerate}[(a)]
			\item $\begin{aligned}[b]
				\lim_{n \rightarrow \infty} \norm{a_n^{-1} - T_n} = 0,
			\end{aligned}$ \label{thm:B:eq:1}
			\item $\begin{aligned}[b]
				\lim_{n \rightarrow \infty} \norm{a_n^{-1} b_n - Q_n} = 0,
			\end{aligned}$ \label{thm:B:eq:2}
			\item $\begin{aligned}[b]
				\lim_{n \rightarrow \infty} \norm{a_n^{-1} a_{n-1}^* - R_n} = 0,
			\end{aligned}$ \label{thm:B:eq:3}
			\item $\begin{aligned}[b]
				\lim_{n \rightarrow \infty} \left\| \frac{a_n}{\norm{a_n}} - C_n \right\| = 0
			\end{aligned}$ \label{thm:B:eq:4}
	 	\end{enumerate}
	 	for $N$-periodic sequences $(T_n: n \geq 0)$, $(Q_n : n \geq 0)$, $(R_n : n \geq 0)$ and
	 	$(C_n : n \geq 0)$ with $C_n$ invertible. Let $\Lambda$ be the set of $\lambda \in \RR$ such 
	 	that\footnote{   The real part of the operator $X$ is defined by $\sym{X} = \frac{1}{2} (X + X^*)$.}
		 \[
			\calF(\lambda) =  
		 	\sym{
			 	\begin{pmatrix}
			 	0 & -C_{N-1} \\
			 	C_{N-1}^* & 0
			 	\end{pmatrix}
			 	\prod_{i=0}^{N-1}
			 	\begin{pmatrix}
			 		0 & \Id \\
			 		-R_i & \lambda T_i - Q_i
			 	\end{pmatrix}
		 	}
		 \]
		 is a strictly positive or a strictly negative operator on $\calH \oplus \calH$.
		 Then for every compact set $K \subset \Lambda$ there are positive constants $c_1, c_2$ such that 
		 for every generalised eigenvector associated with $\lambda \in K$ and every $n \geq 1$
		 \begin{equation} \label{eq:26}
	 		c_1 (\norm{u_0}^2 + \norm{u_1}^2) \leq \norm{a_n} (\norm{u_{n-1}}^2 + \norm{u_n}^2) 
	 		\leq c_2 (\norm{u_0}^2 + \norm{u_1}^2).
	 	\end{equation}
   	\end{theorem}
	When the Carleman condition is satisfied, the asymptotics \eqref{eq:26} implies the similar conclusion as 
	Theorem~\ref{thm:spec:2}, i.e. $\sigmaP{A} \cap \Lambda = \emptyset$ and 
	$\sigma(A) \supset \overline{\Lambda}$. In the scalar case the subordination theory 
	(see, e.g., \cite{Clark1996}) implies that in fact the spectrum of $A$ is purely absolutely continuous 
	on $\Lambda$. Unfortunately, a~subordination theory for the non-scalar case has not been formulated (but there 
	is some progress, see \cite{CarvalhoMarchetti2010}). We expect that in our case the spectrum of $A$
	is, similarly to the scalar case, purely absolutely continuous of the maximal multiplicity on $\Lambda$.
	
	It is also of interest to obtain a~characterization when the symmetric operator $A$ is \emph{not}
	self-adjoint (see, e.g., \cite{Duran2000, Zagorodnyuk2017}). The following Theorem shows that in 
	the setting of Theorem~\ref{thm:B} the Carleman condition is also necessary to the self-adjointness of $A$.
	\begin{theorem} \label{thm:C}
		Let the assumptions of Theorem~\ref{thm:B} be satisfied with $\Lambda \neq \emptyset$. 
		If \eqref{eq:27} is \emph{not} satisfied, then the conclusion of Theorem~\ref{thm:B} holds for 
		$\Lambda = \CC$. Consequently, for every $z \in \CC$
		\[
			\ker[A^* - z\Id] \simeq \calH.
		\]
		Hence, we have the so-called \emph{complete indeterminate} case. In particular, the symmetric operator 
		$A$ is not self-adjoint but it has self-adjoint extensions.
	\end{theorem}
	
	The estimate implied by Theorem~\ref{thm:C} is useful even in the scalar case (see \cite{BergSzwarc2014}).
	
	The method of the proofs of the presented theorems is based on an extension of the techniques used in
	\cite{Swiderski2016} and \cite{Swiderski2017}. In these articles one examines the positivity or 
	the convergence of sequences of quadratic forms on $\RR^2$ acting on the vector of two consecutive values of 
	a generalised eigenvector~$u$ associated with $\lambda \in \Lambda \subset \RR$, i.e.
	\[
		S_n = \left\langle X_n(\lambda) 
		\begin{pmatrix}
			u_{n-1} \\
			u_n
		\end{pmatrix},
		\begin{pmatrix}
			u_{n-1} \\
			u_n
		\end{pmatrix}
		\right\rangle_{\RR^2},
	\]
	for a~suitably chosen sequence $(X_n(\lambda) : n \geq 0)$, $X_n(\lambda) \in \calB(\RR^2)$. 
	In trying to extend this method one encounters several difficulties.
	
	First of all, what is the right quadratic form for the operator case? One real number should 
	control the norm of generalised eigenvectors, which unlike the scalar case, need not to be real. Moreover,
	the convergence (or at least positivity) should be easily expressible in terms of the recurrence relation.
	What additionally complicates the matter is the fact that in general the parameters $(a_n : n \geq 0)$
	and $(b_n : \geq 0)$, unlike the scalars, are not commuting with each other. 
	The second one need not to be even symmetric. Moreover, because of the fact that the Hilbert space $\calH$ can 
	be arbitrary, we  cannot assume that it is locally compact. This complicates the analysis of the proposed
	quadratic forms.
	
	The second issue concerns the problem how one can express quantitatively the rate of divergence 
	or deviation from the positivity of the parameters. As simple examples of diagonal $a_n$ and $b_n$ show, 
	the divergence of the norms is too coarse. The scaling from Theorem~\ref{thm:B}\eqref{thm:B:eq:4} 
	seems to be a~natural one. However, there are also different possibilities known in the literature 
	(see \cite{Duran2001}).
	
	The article is organized as follows. In Section~\ref{sec:prelim} we present basic notions needed in 
	the rest of the article. In Section~\ref{sec:genEig} we define generalised eigenvectors and prove
	the correspondence of their asymptotic behaviour with the spectral properties of $A$. 
	In Section~\ref{sec:commutatorApproach} we prove Theorem~\ref{thm:A}. Next, in Section~\ref{sec:specialCases}, 
	we present its special cases. In particular, the choice of the parameter sequence $\alpha_n \equiv \Id$ 
	motivates us to define  the notion of $N$-shifted Turán determinants in Section~\ref{sec:TuranDet}. 
	Section~\ref{sec:TuranDet} is devoted to the proof of Theorems \ref{thm:B} and \ref{thm:C}. In 
	Section~\ref{sec:exactAsym} we present the situations when one can compute \emph{exact} asymptotics of $u$. 
	In the scalar case it has applications to the so-called Christoffel functions. Finally, in 
	Section~\ref{sec:examples} we present some examples illustrating the sharpness of the assumptions.
	
\section{Preliminaries} \label{sec:prelim}
In this section we collect some basic notations and properties, which will be needed in the sequel.
\subsection{Operators}
	On the space of bounded operators we consider only the norm topology. In particular, a~sequence 
	$(X_n : n \geq 0)$ converges to $X$ provided
	\[
		\lim_{n \rightarrow \infty} \norm{X_n - X} = 0,
	\]
	where $\norm{\cdot}$ is the operator norm.
	
	For a~sequence of operators $(X_n : n \in \NN)$ and $n_0, n_1 \in \NN$ we set
	\[
      \prod_{k=n_0}^{n_1} X_k = 
       \begin{cases} 
       X_{n_1} X_{n_1 - 1} \cdots X_{n_0} & n_1 \geq n_0, \\
       \Id & \text{otherwise.}
       \end{cases}
	\]
	
	For any bounded operator $X$ we define its \emph{real part} by
	\[
		\sym{X} = \frac{1}{2} (X + X^*).
	\]
	Direct computation shows that for any bounded operator $Y$ one has
	\begin{equation} \label{eq:9}
		Y^* \sym{X} Y = \sym{Y^* X Y}
	\end{equation}
	and
	\begin{equation} \label{eq:11}
		\sym{X + Y} = \sym{X} + \sym{Y}.
	\end{equation}
	Moreover,
	\begin{equation} \label{eq:12}
		\norm{\sym{X}} \leq \norm{X}.
	\end{equation}
	
	For a~number $x \in \RR$ we define its \emph{negative part} by the formula
	\[
		x^- = \max(0, -x).
	\]
	For a~self-adjoint operator $X$ we define $X^-$ by the spectral theorem.
	
	For any bounded operator $X$ we define its \emph{absolute value} by
	\[
		|X| = (X^* X)^{1/2}.
	\]

\subsection{Total variation}
	Given a~positive integer $N$, we define the total $N$-variation $\calV_N$ of a~sequence of vectors 
	$x = \big(x_n : n \in \NN\big)$ from a~vector space $V$ by
	\[
		\calV_N(x) = \sum_{n = 0}^\infty \norm{x_{n + N} - x_n}.
	\]
	Observe that if $(x_n : n \in \NN)$ has a finite total $N$-variation then for each 
	$j \in \{0, \ldots, N-1\}$ a~subsequence $(x_{k N + j} : k \in \NN)$ is a~Cauchy sequence. 
	
	\begin{proposition} \label{prelim:normedVariation}
		If $V$ is a~normed algebra, then
		\[
			\calV_N(x_n y_n : n \in \NN) \leq \sup_{n \in \NN}{\norm{x_n}}\ \calV_N(y_n : n \in \NN) +
				\sup_{n \in \NN}{\norm{y_n}}\ \calV_N(x_n : n \in \NN).
		\]
	\end{proposition}
	\begin{proof}
		Observe
		\[
			x_{n+N} y_{n+N} - x_n y_n = (x_{n+N} - x_n) y_{n+N} + x_{n} (y_{n+N} - y_n).
		\]
		Hence,
		\[
			\norm{x_{n+N} y_{n+N} - x_n y_n} \leq \norm{x_{n+N} - x_n} \norm{y_{n+N}} + 
			\norm{x_{n}} \norm{y_{n+N} - y_n}.
      \]
		Consequently,
		\[
			\norm{x_{n+N} y_{n+N} - x_n y_n} \leq \sup_{m \in \NN} \norm{y_{m}} \norm{x_{n+N} - x_n} + 
			\sup_{m \in \NN} \norm{x_{m}} \norm{y_{n+N} - y_n}.
      \]
      Summing by $n$ the result follows.
	\end{proof}
   
\section{generalised eigenvectors and the transfer matrix} \label{sec:genEig}
	For a~number $z \in \CC$, a~non-zero sequence $u = (u_n \colon n \geq 0)$ 
	will be called a~\emph{generalised eigenvector} provided that it satisfies
	\begin{equation} \label{eq:3}
		a_{n-1}^* u_{n-1} + b_n u_n + a_n u_{n+1} = z u_n, \quad (n \geq 1).
	\end{equation}
	For each non-zero $\alpha \in \calH \oplus \calH$ there is a unique generalised
	eigenvector $u$ such that\footnote{We employ the following notation: $(v_1, v_2)^t = 
	\begin{pmatrix}
		v_1 \\
		v_2
	\end{pmatrix}$.} $(u_0, u_1)^t = \alpha$. If the recurrence relation \eqref{eq:3}
	holds also for $n = 0$, with the convention that $a_{-1} = u_{-1} = 0$, then $u$ is 
	a~\emph{formal eigenvector} of the matrix~$A$ associated with $z$.
	
	For each $z \in \CC$ and $n \in \NN$
	we define the \emph{transfer matrix} $B_n(z)$ by
	\begin{equation} \label{eq:13}
		B_n(z) = 
		\begin{pmatrix}
		0 & \Id \\
		-a_n^{-1} a_{n-1}^* & a_{n}^{-1} (z \Id - b_n)
		\end{pmatrix}, \quad (n > 0).
	\end{equation}
	Then for any generalised eigenvector $u$ corresponding to $z$ we have
	\begin{equation} \label{eq:1}
	    \begin{pmatrix}
	        u_n \\
	        u_{n+1}	
		\end{pmatrix}
	    =
	    B_n(z)
	    \begin{pmatrix}
			u_{n-1}\\
			u_n
	    \end{pmatrix}, \quad (n > 0).
	\end{equation}
	It is easy to verify that
	\begin{equation} \label{eq:2}
		B_n^{-1}(z) = 
		\begin{pmatrix}
			(a_{n-1}^*)^{-1} (z \Id - b_n) & -(a_{n-1}^*)^{-1} a_n \\
			\Id & 0
		\end{pmatrix}.
	\end{equation}
	
	The rest of this section concerns relations between generalised eigenvectors and spectral properties
	of block Jacobi matrices.
	
	The proof of \cite[Lemma 2.1]{Beckermann2001} implies that the adjoint operator to $A$ can be
	described as the restriction of $\calA$ to $\ell^2(\NN; \calH)$, i.e. $A^* x = \calA x$ for $x \in \Dom(A^*)$,
	where 
	\begin{equation} \label{eq:35}
		\Dom(A^*) = \{ x \in \ell^2(\NN; \calH) \colon \calA x \in \ell^2(\NN; \calH)\}
	\end{equation}
	
	The following Proposition is essential in examining properties of $A^*$.
	\begin{proposition} \label{prop:9}
		Let $z \in \CC$. The sequence $u$ satisfies $\calA u = z u$ if and only if
		\begin{equation}
	        \begin{gathered} \label{eq:36}
	            u_0 \in \calH, \quad u_1 = a_0^{-1} (z \Id - b_0) u_0, \\
	        	a_{n-1}^* u_{n-1} + b_n u_n + a_n u_{n+1} = z u_n \quad (n \geq 1).
		    	\end{gathered}
		\end{equation}
	\end{proposition}
	\begin{proof}
		It immediately follows from the direct computations.
	\end{proof}

	The following Corollary describes some of the situations when we can describe the deficiency spaces of 
	the operator $A$ explicitly.
	\begin{corollary} \label{cor:4}
		Let $z \in \CC$. If every generalised eigenvector associated with $z$ belongs to $\ell^2(\NN; \calH)$,
		then 
		\begin{equation} \label{eq:34}
			\ker[A^* - z \Id] \simeq \calH.
		\end{equation}
		In particular, if \eqref{eq:34} is satisfied for $z = \pm i$, then
		the symmetric operator $A$ is \emph{not} self-adjoint, but it has self-adjoint extensions.
	\end{corollary}
	\begin{proof}
		Observe that the space $\ker[A^* - z \Id]$ is a~Hilbert space. Indeed, since 
		$\ker[A^* - z \Id] = \Imag{A - \overline{z} \Id}^\perp$ (see, e.g., \cite[formula (7.1.45)]{Simon2017}) 
		it is a~closed subspace of $\ell^2(\NN; \calH)$. 
		
		Define the operator $T : \ker[A^* - z \Id] \rightarrow \calH$ by $T u = u_0$. Then by \eqref{eq:36} 
		$T u = 0$ implies $u=0$, hence, $T$ is injective. To prove the surjectivity take 
		$u_0 \in \calH \setminus \{ 0 \}$, then the sequence $u$ defined by \eqref{eq:36} is a~generalised
		eigenvector associated with $z$. Therefore, it belongs to $\ell^2(\NN; \calH)$. Hence, by \eqref{eq:35} 
		$u \in \Dom(A^*)$, and consequently, $T$ is surjective. Since the mapping $T$ is a~contraction, it is 
		a bounded linear bijection. By the inverse mapping theorem the operator $T$ is a~linear isomorphism.
		
		The assertion about the self-adjoint extensions of $A$ follows from von Neumann’s Extension Theorem 
		(see, e.g., \cite[Theorem 7.4.1]{Simon2017}).
	\end{proof}
	
	\begin{remark}
		The proof of \cite[Theorem 1]{Kostyuchenko1998} shows that the same conclusion holds if every generalised
		eigenvector associated with $z=0$ belongs to $\ell^2(\NN; \calH)$. As it was pointed 
		out in \cite{Braeutigam2016} the formulation of \cite[Theorem 1]{Kostyuchenko1998} has a~typo.
	\end{remark}

	The following Proposition is an adaptation of \cite[Proposition 2.1]{Swiderski2016}. 
	We include it for the sake of self-containment.
   \begin{proposition} \label{prop:3}
      Let $z \in \CC$. If every generalised eigenvector $u$ associated with $z$ does 
      not belong to $\ell^2(\NN; \calH)$ then $z \notin \sigmaP{A^*}$ and 
      $z \in \sigma(A^*)$.
   \end{proposition}
   \begin{proof}
     Let $u \neq 0$ be such that $\calA u = z u$, then by Proposition~\ref{prop:9} $u$ is 
     a~generalised eigenvector associated with $z$. By the assumption $u \notin \ell^2(\NN; \calH)$. 
     Therefore, $u \notin \Dom(A^*)$, and consequently, $z \notin \sigmaP{A^*}$.
      
     Observe that the vector $u$ such that $(\calA - z \Id) u = \delta_0 v$, where $0 \neq v \in \calH$
     has to satisfy the following recurrence relation
     \[
        \begin{gathered}
	        b_0 + a_0 u_1 = z u_0 + v \\
    	 	a_{n-1}^* u_{n-1} + b_n u_n + a_n u_{n+1} = z u_n \quad (n \geq 1)
     	\end{gathered}
     \]
     Hence $u$ is a~generalised eigenvector, thus $u \notin \ell^2(\NN; \calH)$. Therefore, 
     $u \notin \Dom(A^*)$, and consequently, the operator $A^* - z \Id$ is not surjective, 
     i.e. $z \in \sigma(A^*)$.
   \end{proof}

	\begin{remark}
		In the scalar case, if the assumptions of Proposition~\ref{prop:3} are satisfied for $z=0$, then
		the operator $A$ is self-adjoint. We expect the same behaviour for every $\calH$.
	\end{remark}

\section{A commutator approach} \label{sec:commutatorApproach}
	The aim of this Section is to prove the following Theorem. 
   	\begin{theorem} \label{thm:A}
      Let $A$ be a~Jacobi matrix. Assume that there is a~sequence $(\alpha_n : n \geq 0)$ of 
      elements from $\calB(\calH)$ such that
	 \begin{enumerate}[(a)]
		\item $\begin{aligned}[b]
			\sum_{n=1}^\infty \frac{\norm{\sym{\alpha_{n+1} a_{n+1}^* - 
			a_n^* a_{n-1}^{-1} \alpha_{n-1} a_n}^-}}{\norm{\alpha_n a_n^*}} < \infty,
		\end{aligned}$ \label{thm:A:eq:1}
		\item $\begin{aligned}[b]
			\sum_{n=1}^\infty \frac{\norm{a_{n-1}^{-1} \alpha_{n-1} a_n - \alpha_n}}
			{\norm{\alpha_n a_n^*}} < \infty,
		\end{aligned}$ \label{thm:A:eq:2}
		\item $\begin{aligned}[b]
			\sum_{n=0}^\infty \frac{\norm{\alpha_n b_{n+1} - b_n a_{n-1}^{-1} \alpha_{n-1} a_n}}
			{\norm{\alpha_n a_n^*}} < \infty,
		\end{aligned}$ \label{thm:A:eq:3}
		\item $\begin{aligned}[b]
			\sum_{n=0}^\infty \frac{1}{\norm{\alpha_n a_n^*}} = \infty.
		\end{aligned}$ \label{thm:A:eq:4}
	 \end{enumerate}
	 Let $\Lambda$ be the set of $\lambda \in \RR$ such that the following limit exists 
	 in the norm and defines a~strictly positive operator on $\calH \oplus \calH$
	 \[
	 	C(\lambda) = \lim_{n \rightarrow \infty} \frac{1}{\norm{\alpha_n a_n^*}} 
	 	\sym{
			\begin{pmatrix}
				\alpha_n a_{n}^* & -(\lambda \Id - b_n) a_{n-1}^{-1} \alpha_{n-1} a_n \\
				0 & a_n^* a_{n-1}^{-1} \alpha_{n-1} a_n
			\end{pmatrix}
		}.
	 \]
	Then $\sigmaP{A^*} \cap \Lambda = \emptyset$ and $\sigma(A^*) \supset \overline{\Lambda}$.
   	\end{theorem}
	
	Given sequence $(\alpha_n : n \geq 0)$ of elements from $\calB(\calH)$ and $\lambda \in \RR$ we define 
	a sequence of binary quadratic forms $Q^\lambda$ on $\calH \oplus \calH$ by the formula
	\[
		Q_n^\lambda (v) = \frac{1}{\norm{\alpha_n a_n^*}} \left\langle 
		\sym{
			\begin{pmatrix}
				\alpha_{n-1} a_{n-1}^* & -\alpha_{n-1} (\lambda \Id - b_n) \\
				0 & \alpha_n a_n^*
			\end{pmatrix} 
		} v, v \right\rangle.
	\]
	Moreover, we define the sequence of functions by the formula
	\begin{equation} \label{eq:8}
		S_n(\alpha, \lambda) = \norm{\alpha_n a_n^*} Q_n^\lambda
		\lrparen{
		\begin{pmatrix} 
			u_{n-1} \\
			u_n
		\end{pmatrix}},
	\end{equation}
	where $u$ is the generalised eigenvector corresponding to $\lambda$ such that 
	$(u_0, u_1)^t = \alpha \in \calH \oplus \calH$.
	
	The first proposition provides a~different representation of $S_n$.
	\begin{proposition} \label{prop:4}
		An alternative formula for $S_n$ is
		\[
			S_n(\alpha, \lambda) = \left\langle
			\sym{
				\begin{pmatrix}
					\alpha_n a_n^* & -(\lambda \Id - b_n) a_{n-1}^{-1} \alpha_{n-1} a_n \\
					0 & a_n^* a_{n-1}^{-1} \alpha_{n-1} a_n
				\end{pmatrix}
			}
			\begin{pmatrix}
				u_{n} \\
				u_{n+1}
			\end{pmatrix},
			\begin{pmatrix}
				u_{n} \\
				u_{n+1}
			\end{pmatrix}
			\right\rangle.
		\]
	\end{proposition}
	\begin{proof}
		By \eqref{eq:1} one has
		\begin{multline*}
			S_n(\alpha, \lambda) = \left\langle 
			\sym{
				\begin{pmatrix}
					\alpha_{n-1} a_{n-1}^* & -\alpha_{n-1} (\lambda \Id - b_n) \\
					0 & \alpha_n a_n^*
				\end{pmatrix}
			}
			B^{-1}_n(\lambda) 
			\begin{pmatrix} 
				u_{n} \\
				u_{n+1}
			\end{pmatrix}, 
			B^{-1}_n(\lambda) 
			\begin{pmatrix} 
				u_{n} \\
				u_{n+1}
			\end{pmatrix}
			\right\rangle \\
			=
			\left\langle
			(B_n^{-1}(\lambda))^{*} 
			\sym{
				\begin{pmatrix}
					\alpha_{n-1} a_{n-1}^* & -\alpha_{n-1} (\lambda \Id - b_n) \\
					0 & \alpha_n a_n^*
				\end{pmatrix} 
			}
			B_n^{-1}(\lambda)
			\begin{pmatrix} 
				u_{n} \\
				u_{n+1}
			\end{pmatrix},
			\begin{pmatrix} 
				u_{n} \\
				u_{n+1}
			\end{pmatrix}			
			\right\rangle.
		\end{multline*}
		Then formula \eqref{eq:2} implies
		\begin{multline*}
			(B_n^{-1}(\lambda))^{*} 
			\begin{pmatrix}
				\alpha_{n-1} a_{n-1}^* & -\alpha_{n-1} (\lambda \Id - b_n) \\
				0 & \alpha_n a_n^*
			\end{pmatrix} 
			B_n^{-1}(\lambda) \\
			= 
			\begin{pmatrix} 
				(\lambda \Id - b_n) a_{n-1}^{-1} & \Id \\
				- a_{n}^* a_{n-1}^{-1} & 0
			\end{pmatrix}
			\begin{pmatrix}
				0 & -\alpha_{n-1} a_n \\
				\alpha_n a_n^* & 0
			\end{pmatrix} \\
			= 
			\begin{pmatrix}
				\alpha_n a_n^* & -(\lambda \Id - b_n) a_{n-1}^{-1} \alpha_{n-1} a_n \\
				0 & a_{n}^* a_{n-1}^{-1} \alpha_{n-1} a_n
			\end{pmatrix}.
		\end{multline*}
		Hence, by formula~\eqref{eq:9}
		\[
			S_n(\alpha, \lambda) = \left\langle
			\sym{
				\begin{pmatrix}
					\alpha_n a_n^* & -(\lambda \Id - b_n) a_{n-1}^{-1} \alpha_{n-1} a_n \\
					0 & a_n^* a_{n-1}^{-1} \alpha_{n-1} a_n
				\end{pmatrix}
			}
			\begin{pmatrix}
				u_{n} \\
				u_{n+1}
			\end{pmatrix},
			\begin{pmatrix}
				u_{n} \\
				u_{n+1}
			\end{pmatrix}
			\right\rangle
		\]
		what ends the proof.
	\end{proof}
	
	The next proposition provides assumptions on the quadratic form under which
	it controls the norm of generalised eigenvectors.
	\begin{proposition} \label{prop:1}
		Let $\Lambda$ be the set of $\lambda \in \RR$ such that the following limit exists in the operator norm
		and defines a~strictly positive operator
		 \[
		 	C(\lambda) = \lim_{n \rightarrow \infty} \frac{1}{\norm{\alpha_n a_n^*}} 
		 	\sym{
				\begin{pmatrix}
					\alpha_n a_{n}^* & -(\lambda \Id - b_n) a_{n-1}^{-1} \alpha_{n-1} a_n \\
					0 & a_n^* a_{n-1}^{-1} \alpha_{n-1} a_n
				\end{pmatrix}
			}.
		 \]
		Then for every $\lambda \in \Lambda$ there is an integer $N$ and positive constants $c_1, c_2$ 
		such that for every generalised eigenvector $u$ associated with $\lambda$ and 
		$0 \neq \alpha \in \calH \oplus \calH$
		\[
			c_1 \norm{\alpha_n a_n^*} (\norm{u_{n}}^2 + \norm{u_{n+1}}^2) \leq 
			S_n(\alpha, \lambda) \leq 
			c_2 \norm{\alpha_n a_n^*} (\norm{u_{n}}^2 + \norm{u_{n+1}}^2), \quad (n \geq N)
		\]
	\end{proposition}
	\begin{proof}
		Fix $\lambda \in \Lambda$. Let
		\[
			\mu^{\textrm{min}}_n = \min \sigma(Z_n), \quad 
			\mu^{\textrm{max}}_n = \max \sigma(Z_n),
		\]
		where
		\[
			Z_n = \frac{1}{\norm{\alpha_n a_n^*}}
			\sym{
			\begin{pmatrix}
				\alpha_{n-1} a_{n-1}^* & -\alpha_{n-1} (\lambda \Id - b_n) \\
				0 & \alpha_n a_n^*
			\end{pmatrix} 
		}.
		\]
		Hence,
		\[
			\mu^{\textrm{min}}_n
			\leq \frac{S_n(\alpha, \lambda)}{\norm{\alpha_n a_n^*}(\norm{u_{n}}^2 + \norm{u_{n+1}}^2)} 
			\leq \mu^{\textrm{max}}_n.
		\]
		But from the definition of $C(\lambda)$ we have
		\[
			\lim_{n \rightarrow \infty} \mu^{\textrm{min}}_n = \min \sigma(C(\lambda)), \quad 
			\lim_{n \rightarrow \infty} \mu^{\textrm{max}}_n = \max \sigma(C(\lambda))
		\]
		which are positive numbers. Therefore, there is $N$ and $c_1, c_2>0$ such that for every $n \geq N$
		\[
			c_1
			\leq \frac{S_n(\alpha, \lambda)}{\norm{\alpha_n a_n^*}(\norm{u_{n}}^2 + \norm{u_{n+1}}^2)} 
			\leq c_2
		\]
		and the proof is complete.
	\end{proof}
	
	The next corollary together with Proposition~\ref{prop:3} suggest the method of proving that every 
	$\lambda \in \Lambda$ is not an eigenvalue of $A$ but belongs to $\sigma(A)$.
	\begin{corollary} \label{cor:1}
		Under the assumptions of Proposition~\ref{prop:1}, together with
		\[
			\sum_{n = 0}^\infty \frac{1}{\norm{\alpha_n a_n^*}} = \infty
		\]
		if
		\[
			\liminf_{n \rightarrow \infty} S_n(\alpha, \lambda) > 0,
		\]
		then $u$ does not belong to $\ell^2(\NN; \calH)$.
	\end{corollary}
	\begin{proof}
		By Proposition~\ref{prop:1}
		\[
			\frac{S_n(\alpha, \lambda)}{c_2 \norm{\alpha_n a_n^*}} \leq \norm{u_n}^2 + \norm{u_{n+1}}^2 
		\]
		for a~positive constant $c_2$. Therefore, there exists a~constant $c>0$ such that
		\[
			\frac{c}{\norm{\alpha_n a_n^*}} \leq \norm{u_n}^2 + \norm{u_{n+1}}^2,
		\]
		which cannot be summable.
	\end{proof}
	
	The following Lemma is the main algebraic part of the proof of Theorem~\ref{thm:A}.
	\begin{lemma} \label{lem:2}
		Let $u$ be a~generalised eigenvector associated with $\lambda \in \RR$ and 
		$\alpha \in \calH \oplus \calH$. Then
		\begin{multline*}
			\frac{[S_{n+1}(\alpha, \lambda) - S_n(\alpha, \lambda)]^-}{\norm{u_n}^2 + \norm{u_{n+1}}^2} 
			\leq 
			\norm{\sym{\alpha_{n+1} a_{n+1}^* - a_{n}^* a_{n-1}^{-1} \alpha_{n-1} a_n}^-} \\
			+ 
			|\lambda| \norm{a_{n-1}^{-1} \alpha_{n-1} a_n - \alpha_n} +
			\norm{\alpha_n b_{n+1} - b_n a_{n-1}^{-1} \alpha_{n-1} a_n}.
		\end{multline*}
	\end{lemma}
	\begin{proof}
		By Proposition~\ref{prop:4} and formula~\eqref{eq:8} we have
		\[
			S_{n+1}(\alpha, \lambda) - S_n(\alpha, \lambda) =
			\left\langle \sym{C_n^\lambda}
			\begin{pmatrix}
				u_{n} \\
				u_{n+1}
			\end{pmatrix},
			\begin{pmatrix}
				u_{n} \\
				u_{n+1}
			\end{pmatrix}			
			\right\rangle
		\]
		for
		\begin{multline*}
			C_n^\lambda = 
			\begin{pmatrix}
				\alpha_{n} a_{n}^* & -\alpha_{n} (\lambda \Id - b_{n+1}) \\
				0 & \alpha_{n+1} a_{n+1}^*
			\end{pmatrix} 
			-
			\begin{pmatrix}
				\alpha_n a_n^* & -(\lambda \Id - b_n) a_{n-1}^{-1} \alpha_{n-1} a_n \\
				0 & a_n^* a_{n-1}^{-1} \alpha_{n-1} a_n
			\end{pmatrix} \\
			=
			\begin{pmatrix}
				0 & (\lambda \Id - b_n) a_{n-1}^{-1} \alpha_{n-1} a_n - \alpha_n (\lambda \Id - b_{n+1}) \\
				0 & \alpha_{n+1} a_{n+1}^* - a_{n}^* a_{n-1}^{-1} \alpha_{n-1} a_n
			\end{pmatrix}.
		\end{multline*}
		Hence,
		\begin{multline*}
			S_{n+1}(\alpha, \lambda) - S_n(\alpha, \lambda) = 
			\langle \sym{\alpha_{n+1} a_{n+1}^* - a_{n}^* a_{n-1}^{-1} \alpha_{n-1} a_n} u_{n+1}, 
			u_{n+1} \rangle_\calH \\
			+ \lambda \Re \langle (a_{n-1}^{-1} \alpha_{n-1} a_n - \alpha_n) u_{n+1}, u_n \rangle_\calH
			+ \Re \langle (\alpha_n b_{n+1} - b_n a_{n-1}^{-1} \alpha_{n-1} a_n) u_{n+1}, u_n \rangle_\calH.
		\end{multline*}
		By the Schwarz inequality the result follows.
	\end{proof}

	We are ready to prove Theorem~\ref{thm:A}.
	\begin{proof}[Proof of Theorem~\ref{thm:A}]
		By virtue of Corollary~\ref{cor:1} and Proposition~\ref{prop:3} it is enough to show that
		$\liminf_n S_n(\alpha, \lambda) > 0$ for every $\lambda \in \Lambda$ and a~non-zero 
		$\alpha \in \calH \oplus \calH$.
		
		Fix $\lambda \in \Lambda$ and a~non-zero $\alpha \in \calH \oplus \calH$.
		By Proposition~\ref{prop:1} there exists $N$ such that for every $n \geq N$ holds 
		$S_n(\alpha, \lambda) > 0$. 
		Let us define 
		\[
			F_n(\alpha, \lambda) = 
			\frac{S_{n+1}(\alpha, \lambda) - S_n(\alpha, \lambda)}{S_n(\alpha, \lambda)}.
		\]
		Then 
		\[
			\frac{S_{n+1}(\alpha, \lambda)}{S_n(\alpha, \lambda)} = 1 + F_n(\alpha, \lambda),
		\]
		and consequently,
		\[
			\frac{S_n(\alpha, \lambda)}{S_N(\alpha, \lambda)} = \prod_{k=N}^{n-1} (1 + F_k(\alpha, \lambda)).
		\]
		Hence,
		\begin{equation} \label{eq:10}
			\sum_{k=N}^\infty [F_k(\alpha, \lambda)]^- < \infty
		\end{equation}
		implies $\liminf_n S_n(\alpha, \lambda) > 0$. By Proposition~\ref{prop:1}
		\[
			S_n(\alpha, \lambda) \geq c^{-1} \norm{\alpha_n a_n^*} (\norm{u_n}^2 + \norm{u_{n+1}^2})
		\]
		for some constant $c>0$. Hence, by Lemma~\ref{lem:2}
		\begin{multline*}
			[F_n(\alpha, \lambda)]^- \leq \frac{c}{\norm{\alpha_n a_n^*}} 
			(\norm{\sym{\alpha_{n+1} a_{n+1}^* - a_{n}^* a_{n-1}^{-1} \alpha_{n-1} a_n}^-} \\
			+ 
			|\lambda| \norm{a_{n-1}^{-1} \alpha_{n-1} a_n - \alpha_n} +
			\norm{\alpha_n b_{n+1} - b_n a_{n-1}^{-1} \alpha_{n-1} a_n}),
		\end{multline*}
		which is summable by assumptions \eqref{thm:A:eq:1}, \eqref{thm:A:eq:2} and \eqref{thm:A:eq:3}. 
		This shows \eqref{eq:10}. The proof is complete.
	\end{proof}

\section{Special cases of Theorem~\ref{thm:A}} \label{sec:specialCases}
In this section we show several choices of the sequence $(\alpha_n : n \geq 0)$. In this way
we show the flexibility of our approach. For the simplification of the condition for $C(\lambda)$
we assume that the sequence $(a_n : n \geq 0)$ tends to infinity, i.e.
\[
	\lim_{n \rightarrow \infty} \norm{a_n^{-1}} = 0.
\]
This condition implies that $C(\lambda)$ does not depend on $\lambda$.

The first theorem is an extension of \cite[Theorem 1.6]{JanasMoszynski2002} to the operator case. 
Since Section~\ref{sec:TuranDet} is devoted to the proof of a~far reaching extension of this result,
we omit the details.
\begin{theorem} \label{thm:spec:1}
	Assume
	\begin{enumerate}[(a)]
		\item $\begin{aligned}[b]
			\sum_{n=1}^\infty \norm{a_{n+1}^* a_n^{-1} - a_n^* a_{n-1}^{-1}} < \infty,
		\end{aligned}$ \label{thm:spec:1:eq:1}
		
		\item $\begin{aligned}[b]
			\sum_{n=1}^\infty \norm{a_{n-1}^{-1} - a_n^{-1}} < \infty,
		\end{aligned}$ \label{thm:spec:1:eq:2}
		
		\item $\begin{aligned}[b]
			\sum_{n=0}^\infty \norm{b_{n+1} a_n^{-1} - b_n a_{n-1}^{-1}} < \infty,
		\end{aligned}$ \label{thm:spec:1:eq:3}
		
		\item $\begin{aligned}[b]
			\sum_{n=0}^\infty \frac{1}{\norm{a_n}} = \infty
		\end{aligned}$ \label{thm:spec:1:eq:4}
	\end{enumerate}
	and $C(\lambda)$ defined for $\alpha_n \equiv \Id$ is a~positive operator on $\calH \oplus \calH$. 
	Then the assumptions of Theorem~\ref{thm:A} are satisfied.
\end{theorem}

We are ready to prove Theorem~\ref{thm:spec:2}. Let us note that this result is a~vector valued version 
of \cite[Theorem 4.3]{Swiderski2016}. In the scalar case it has far reaching applications 
(see \cite[Section 5]{Swiderski2016}).

\begin{proof}[Proof of Theorem~\ref{thm:spec:2}]
	Take $\alpha_n = a_n$. It is sufficient to show that $\Lambda = \RR$. We have
	\[
		C(\lambda) = 
		\lim_{n \rightarrow \infty} \frac{1}{\norm{a_n a_n^*}} 
	 	\sym{
			\begin{pmatrix}
				a_n a_{n}^* & -(\lambda \Id - b_n) (a_n^*)^{-1} a_n^* a_n \\
				0 & a_n^* a_n
			\end{pmatrix}
		}
		= \sym{
		\begin{pmatrix}
			C C^* & 0 \\
			0 & C^* C
		\end{pmatrix}
		},
	\]
	which is clearly positive for $\lambda \in \RR$. Hence, $\Lambda = \RR$.
\end{proof}

To formulate the last example we need a~definition. Let
\begin{equation} \label{eq:30}
	\log^{(0)}(x) = x, \qquad \log^{(i+1)}(x) = \log(\log^{(i)}(x)) \quad (i \geq 0)
\end{equation}
and
\[
	g_j(x) = \prod_{i=1}^j \log^{(i)}(x).
\]
The following Theorem is a~vector valued version of \cite[Theorem 4.3]{Swiderski2016} and its proof is 
inspired by the techniques employed in the proof of \cite[Theorem 3]{Janas2014}.
\begin{theorem} \label{thm:spec:3}
	Assume that for positive integers $K, N$ and a~non-negative summable sequence~$c_n$
	\begin{enumerate}[(a)]
    	\item $\begin{aligned}[b]
			\lim_{n \rightarrow \infty} a_n^{-1} = 0,
		\end{aligned}$ \label{thm:spec:3:eq:1}
		
		\item $\begin{aligned}[b]
			(1 - c_n) \Id \leq | (a_{n-1}^*)^{-1} a_n | \leq \left( 1 + \frac{1}{n} + 
			\sum_{j=1}^K \frac{1}{n g_j(n)} + c_n \right) \Id \quad
		\end{aligned}$ 
		for $n > N$, \label{thm:spec:3:eq:2}
		
		\item the sequence $(b_n : n \geq 0)$ is bounded and 
		$\begin{aligned}[b]
			\sum_{n=0}^\infty \norm{a_n^{-1} b_n - b_{n+1} a_n^{-1} } < \infty,
		\end{aligned}$ \label{thm:spec:3:eq:3}
		
		\item $\begin{aligned}[b]
			\sum_{n=1}^\infty \frac{\lVert a_n^{-1} \rVert }{n} < \infty.
		\end{aligned}$	\label{thm:spec:3:eq:4}
    \end{enumerate}
    Then the assumptions of Theorem~\ref{thm:A} are satisfied with $\Lambda = \RR$.
\end{theorem}
\begin{proof}
	We can assume that $\log^{(K)}(N) > 0$. Let
    \[
        \alpha_n = 
		\begin{cases}
			\Id & \text{ for } n < N, \\
			n g_K(n) (a_n^*)^{-1} & \text{ otherwise}
		\end{cases}
	\]
	We have to compute the set $\Lambda$ and check the assumptions \eqref{thm:A:eq:1}, 
	\eqref{thm:A:eq:2}, \eqref{thm:A:eq:3} of Theorem~\ref{thm:A}.
	
	Let us begin with the computation of $\Lambda$. We have
	\begin{multline*}
		\frac{1}{\norm{\alpha_n a_n^*}} 
			\begin{pmatrix}
				\alpha_n a_{n}^* & -(\lambda \Id - b_n) a_{n-1}^{-1} \alpha_{n-1} a_n \\
				0 & a_n^* a_{n-1}^{-1} \alpha_{n-1} a_n
			\end{pmatrix}
		 \\ 
		=
			\begin{pmatrix}
				\Id & -\frac{(n-1) g_K(n-1)}{n g_K(n)} (\lambda \Id - b_n) 
					(a_n^*)^{-1} |(a_{n-1}^*)^{-1} a_n|^2 \\
				0 & \frac{(n-1) g_K(n-1)}{n g_K(n)} |(a_{n-1}^*)^{-1} a_n|^2
			\end{pmatrix}
	\end{multline*}
	which by the hypotheses \eqref{thm:spec:3:eq:1} and \eqref{thm:spec:3:eq:2} tends to
	\[
		\begin{pmatrix}
			\Id & 0 \\
			0 & \Id
		\end{pmatrix},
	\]
	which is clearly a positive operator on $\calH \oplus \calH$ for any $\lambda \in \RR$. Hence, $\Lambda = \RR$.
	
	Let us show the assumption \eqref{thm:A:eq:1}. We have
	\begin{multline*}
		\frac{\alpha_{n+1} a_{n+1}^* - a_n^* a_{n-1}^{-1} \alpha_{n-1} a_n}{\norm{\alpha_n a_n^*}} =
        \frac{(n+1) g_K(n+1)}{n g_K(n)} \Id 
        - \frac{(n-1) g_K(n-1)}{n g_K(n)} | (a_{n-1}^*)^{-1} a_n |^2 \\ \geq 
        \left( \frac{(n+1) g_K(n+1)}{n g_K(n)} - \frac{(n-1) g_K(n-1)}{n g_K(n)} \left( 1 + \frac{1}{n} 
        + \sum_{j=1}^K \frac{1}{n g_j(n)} + c_n \right)^2 \right) \Id.
	\end{multline*}
	The above expression has been estimated in the proof of \cite[Theorem~4.3]{Swiderski2016}.
	
	Next, since
	\[
		\alpha_n b_{n+1} - b_n a_{n-1}^{-1} \alpha_{n-1} a_n = \alpha_n b_{n+1} - b_{n} \alpha_n 
		+ b_n(\alpha_n - a_{n-1}^{-1} \alpha_{n-1} a_n),
	\]
	the hypothesis~\eqref{thm:spec:3:eq:3} implies that the assumption~\eqref{thm:A:eq:2} will be satisfied 
	if we show that the assumption~\eqref{thm:A:eq:3} holds.
	
	We have
	\[
		\frac{a_{n-1}^{-1} \alpha_{n-1} a_n - \alpha_n}{\norm{\alpha_n a_n^*}} = 
		\frac{(n-1) g_K(n-1)}{n g_K(n)} a_{n-1}^{-1} (a_{n-1}^*)^{-1} a_n - (a_n^*)^{-1} = 
		a_{n-1}^{-1} T_n,
	\]
	where
	\[
		T_n = \frac{(n-1) g_K(n-1)}{n g_K(n)} W_n^* - W_n^{-1}, \quad W_n = a_n^* a_{n-1}^{-1}.
	\]
	By virtue of the hypothesis \eqref{thm:spec:3:eq:4}, the assumption~\eqref{thm:A:eq:3} will be satisfied 
	as long as
	\begin{equation} \label{eq:5}
		\norm{T_n} \leq c \left( \frac{1}{n} + c_n' \right)
	\end{equation}
	for a~constant $c>0$ and a~non-negative summable sequence $(c_n' : n \geq 0)$. Because
	\[
		T_n T_n^* = \left( \frac{(n-1) g_K(n-1)}{n g_K(n)} \right)^2 W_n^* W_n 
		- 2 \frac{(n-1) g_K(n-1)}{n g_K(n)} \Id + (W_n^* W_n)^{-1},
	\]
	the non-negativity of $T_n T_n^*$ and $\norm{T_n T_n^*} = \norm{T_n}^2$, 
	the inequality \eqref{eq:5} will be satisfied if
	\[
		T_n T_n^* \leq c^2 \left( \frac{1}{n} + c_n' \right)^2 \Id.
	\]
	The spectral theorem applied to $W_n^* W_n$ implies that the above inequality will be satisfied if
	\begin{equation} \label{eq:6}
		\left( \frac{(n-1) g_K(n-1)}{n g_K(n)} \right)^2 \lambda_n 
		- 2 \frac{(n-1) g_K(n-1)}{n g_K(n)} + \lambda_n^{-1} \leq \left( \frac{1}{n} + c_n' \right)^2
	\end{equation}
	for every $\lambda_n \in \sigma(W_n^* W_n)$, which by the hypothesis~\eqref{thm:spec:3:eq:2} corresponds to
	\[
		\lambda_n \in 
		\left[ (1-c_n)^2, \left( 1 + \frac{1}{n} + \sum_{j=1}^K \frac{1}{n g_j(n)} + c_n \right)^2 \right].
	\]
	But
	\[
		\left( \frac{(n-1) g_K(n-1)}{n g_K(n)} \right)^2 \lambda_n 
		- 2 \frac{(n-1) g_K(n-1)}{n g_K(n)} 
		+ \lambda^{-1}_n = \left( \frac{(n-1) g_K(n-1)}{n g_K(n)} \sqrt{\lambda_n} 
		- \frac{1}{\sqrt{\lambda_n}} \right)^2
	\]
	and the above expression has been estimated in the proof of \cite[Theorem~4.3]{Swiderski2016}. 
	This shows \eqref{eq:6} and ends the proof.
\end{proof}

\section{Turán determinants} \label{sec:TuranDet}
	Let us note that for $\calH = \RR$ the expression $S_n$ for $\alpha_n \equiv \Id$ (see \eqref{eq:8}) 
	is known as the $N$-shifted \emph{Turán determinant} (see \cite{VanAssche1991}). Hence, 
	Theorem~\ref{thm:spec:1} motivates us to the following construction.
	Fix a~positive integer $N$ and a~Jacobi matrix $A$. Let us define a~sequence of quadratic 
	forms $Q^z$ on $\calH \oplus \calH$ by the formula
	\begin{equation} \label{eq:16}
		Q_n^z(v) = \frac{1}{\norm{a_{n+N-1}}} \left\langle 
		\sym{
			\begin{pmatrix}
				a_{n+N-1} & 0 \\
				0 & a_{n+N-1}^*
			\end{pmatrix}
			E X_n(z)
		} v, v \right\rangle,
	\end{equation}
	where
	\[
		X_n(z) = \prod_{j = n}^{n+N-1} B_j(z)
		\quad\text{and}\quad
	    E = 
	    \begin{pmatrix}
	    	0 & -\Id \\
	        \Id & 0
	    \end{pmatrix}.
	\]
	Then we define the $N$-shifted Turán determinants by
	\begin{equation} \label{eq:7}
		S_n(\alpha, z) = \norm{a_{n+N-1}} Q_n^z 
		\lrparen{
		\begin{pmatrix} 
			u_{n-1} \\
			u_n
		\end{pmatrix}},
	\end{equation}
	where $u$ is the generalised eigenvector corresponding to $z \in \CC$ such that 
	$(u_0, u_1)^t = \alpha \in \calH \oplus \calH$. 
	
	The rest of this section is devoted to the analysis of the sequence $S_n$. Since the proof
	of the uniform convergence of $S_n$ is quite involved, we divide it into 3 subsections. The method
	used here is an adaptation of the techniques employed in \cite{Swiderski2017}.

\subsection{Almost uniform non-degeneracy}
	Let $\Lambda$ be a~subset of $\CC$. In this section we consider the family
	$\{ Q^z : z \in \Lambda \}$ defined in \eqref{eq:16}.
	
	We say that $\{ Q^z : z \in \Lambda \}$ is \emph{uniformly non-degenerated} on 
	$K \subset \Lambda$ if there are $c \geq 1$ and $M \geq 1$ such that for all 
	$v \in \calH \oplus \calH$, $z \in K$ and $n \geq M$ 
	\[
		c^{-1} \norm{v}^2 \leq \abs{Q_n^z(v)} \leq c \norm{v}^2.
	\]
	We say that $\{Q^z : z \in \Lambda\}$ is \emph{almost uniformly non-degenerated} 
	on $\Lambda$ if it is uniformly non-degenerated on each compact subset of $\Lambda$.

	We begin with two simple auxiliary results which will be needed in the proof of the non-degeneracy of
	the considered quadratic forms.
	\begin{lemma} \label{lem:1}
		For every $n$ and $\lambda \in \RR$ one has
		\[
			\begin{pmatrix}
				a_n & 0 \\
				0 & a_n^*
			\end{pmatrix}
			E B_n(\lambda) 
			= 
			[B_n^{-1}(\lambda)]^* 
			\begin{pmatrix}
				a_{n-1} & 0 \\
				0 & a_{n-1}^*
			\end{pmatrix}
			E.
		\]
	\end{lemma}
	\begin{proof}
		Using \eqref{eq:2} and \eqref{eq:13} one can compute that both sides are equal to
		\[
			\begin{pmatrix}
				a_{n-1}^* & -(\lambda \Id - b_n) \\
				0 & a_n^*
			\end{pmatrix}
		\]
		and the result follows.
	\end{proof}
	
	\begin{proposition} \label{prop:6}
		Let $N$ be an integer. Assume
		\begin{enumerate}[(a)]
			\item $\begin{aligned}[b]
				\lim_{n \rightarrow \infty} \norm{a_n^{-1} a_{n-1}^* - R_n} = 0,
			\end{aligned}$
			\item $\begin{aligned}[b] 
				\lim_{n \rightarrow \infty} \left\| \frac{a_n}{\norm{a_n}} - C_n \right\| = 0.
			\end{aligned}$
		\end{enumerate}
		for $N$-periodic sequences of invertible operators $R$ and $C$. Then
		\[
			\lim_{n \rightarrow \infty} \left\| \frac{\norm{a_n}}{\norm{a_{n-1}}} \Id 
			- C_n^{-1} C_{n-1}^* R_n^{-1} \right\| = 0.
		\]
		In particular,
		\[
			\lim_{n \rightarrow \infty} \left| \frac{\norm{a_n}}{\norm{a_{n-1}}} - r_n \right| = 0,
		\]
		for a~positive $N$-periodic sequence
		\[
			r_n = \norm{C_n^{-1} C_{n-1}^* R_n^{-1}}.
		\]
	\end{proposition}	
	\begin{proof}
		We have
		\[
			\frac{\norm{a_n}}{\norm{a_{n-1}}} \Id = \left( \frac{a_n}{\norm{a_n}} \right)^{-1} 
			\frac{a_{n-1}^*}{\norm{a_{n-1}}} (a_n^{-1} a_{n-1}^*)^{-1}.
		\]
		Hence, 
		\[
			\lim_{n \rightarrow \infty} \left\| \frac{\norm{a_n}}{\norm{a_{n-1}}} \Id 
			- C_n^{-1} C_{n-1}^* R_n^{-1} \right\| = 0
		\]
		and the result follows.
	\end{proof}
	
	In the next proposition we examine the limiting behaviour of the considered quadratic forms.
	\begin{proposition} \label{prop:5}
		Let $N \geq 1$ be an integer. Assume
   	\begin{enumerate}[(a)]
			\item $\begin{aligned}[b]
				\lim_{n \rightarrow \infty} \norm{a_n^{-1} - T_n} = 0,
			\end{aligned}$ \label{prop:5:eq:1}
			\item $\begin{aligned}[b]
				\lim_{n \rightarrow \infty} \norm{a_n^{-1} b_n - Q_n} = 0,
			\end{aligned}$ \label{prop:5:eq:2}
			\item $\begin{aligned}[b]
				\lim_{n \rightarrow \infty} \norm{a_n^{-1} a_{n-1}^* - R_n} = 0,
			\end{aligned}$ \label{prop:5:eq:3}
			\item $\begin{aligned}[b] 
				\lim_{n \rightarrow \infty} \left\| \frac{a_n}{\norm{a_n}} - C_n \right\| = 0.
			\end{aligned}$ \label{prop:5:eq:4}
		\end{enumerate}
		for $N$-periodic sequences $T, Q, R$ and $C$ such that for every $n$ the operators 
		$R_n$ and $C_n$ are invertible. Then on every compact subset of $\CC$ the sequence 
		$(\norm{X_n(\cdot)} : n \geq 0)$ is uniformly bounded. Moreover, 
		\begin{equation} \label{eq:29}
			\lim_{n \rightarrow \infty}
			\left\| 
				\begin{pmatrix}
					\frac{a_{n+N-1}}{\norm{a_{n+N-1}}} & 0 \\
					0 & \frac{a^*_{n+N-1}}{\norm{a_{n+N-1}}}
				\end{pmatrix}
				E X_n(\cdot) - \calF^n(\cdot)
			\right\|
			= 0
		\end{equation}
		uniformly on compact subsets of $\CC$, where
		\[
			\calF^n(z) =
			\begin{pmatrix}
				C_{n+N-1} & 0 \\
				0 & C_{n+N-1}^*
			\end{pmatrix}
			E
			\prod_{k=n}^{N+n-1}
			\begin{pmatrix}
				0 & \Id \\
				-R_k & z T_k - Q_k
			\end{pmatrix}.
		\]
	\end{proposition}
	\begin{proof}
		Let us define
		\[
			\calX_n(z) = \prod_{j=n}^{n+N-1} \calB_j(z), \quad \text{where} \quad
			\calB_n(z) =
			\begin{pmatrix}
				0 & \Id \\
				-R_n & z T_n - Q_n
			\end{pmatrix}.
		\]	
		We have
		\[
			\norm{B_n(z) - \calB_n(z)} \leq \norm{R_n - a_n^{-1} a_{n-1}^*} 
			+ |z| \norm{a_n^{-1} - T_n} + \norm{Q_n - a_n^{-1} b_n},
		\]
		which tends to $0$ uniformly on compact subsets of $\CC$. Consequently, since every function 
		$B_n(\cdot)$ is continuous, one has
		\[
			\lim_{n \rightarrow \infty} \norm{X_n(\cdot) - \calX_n(\cdot)} = 0
		\]
		uniformly on the compact subsets of $\CC$. In particular, it implies \eqref{eq:29} and the uniform 
		boundedness of $(\norm{X_n(\cdot)} : n \geq 0)$ on every compact subset of $\CC$.
	\end{proof}

	Finally, in the last proposition, we formulate the conditions under which the sequence 
	$\{ Q^z : z \in \Lambda \}$ is almost uniformly non-degenerated.
	\begin{proposition} \label{prop:8}
		Let the assumptions of Proposition~\ref{prop:5} be satisfied. If for every $i \in \NN$ 
		and every $z \in \Lambda$ there is $\varepsilon(i, z) \in \{-1, 1\}$ such that
		\begin{equation} \label{eq:28}
			\varepsilon(i, z) \sym{\calF^i(z)} > 0,
		\end{equation}
		then $(Q^z : z \in \Lambda)$ is almost uniformly non-degenerated. Moreover,
		if $\Lambda \subset \RR$, then the same conclusion follows provided \eqref{eq:28} holds only for $i=0$.
	\end{proposition}
	\begin{proof}
		By \eqref{eq:29} and \eqref{eq:28} we have that for every compact $K \subset \Lambda$ there is 
		a constant $c>0$ such that for $n$ sufficiently large and all $z \in K$
		\[
			\varepsilon(i, z) 
			\sym{
			\begin{pmatrix}
					\frac{a_{n+N-1}}{\norm{a_{n+N-1}}} & 0 \\
					0 & \frac{a^*_{n+N-1}}{\norm{a_{n+N-1}}}
				\end{pmatrix}
				E X_n(z)
			}
			> c \Id.
		\]
		It implies the uniform non-degeneracy of $\{ Q^z : z \in K \}$.
		
		Consider $\lambda \in \RR$. According to Lemma~\ref{lem:1} we have
		\[
			\frac{\norm{a_{n+N}}}{\norm{a_{n+N-1}}}
			\begin{pmatrix}
				\frac{a_{n+N}}{\norm{a_{n+N}}} & 0 \\
				0 & \frac{a_{n+N}^*}{\norm{a_{n+N}}}
			\end{pmatrix}
			E X_{n+1}(\lambda)
			=
			[B_{n+N}^{-1}(\lambda)]^*
			\begin{pmatrix}
				\frac{a_{n+N-1}}{\norm{a_{n+N-1}}} & 0 \\
				0 & \frac{a_{n+N-1}^*}{\norm{a_{n+N-1}}}
			\end{pmatrix}
			E X_n(\lambda) B_n^{-1}(\lambda).
		\]
		Let $n = kN+i$ and let us compute the limit of both sides as $k$ tends to $\infty$. 
		By Propositions \ref{prop:6} and \ref{prop:5} we have
		\[
			r_i 
			\calF^i(\lambda)
			=
			[\calB_i^{-1}(\lambda)]^*
				\calF^{i-1}(\lambda)
			\calB_i^{-1}(\lambda),
		\]
		where
		\[
			\calB_i(\lambda) = 
			\begin{pmatrix}
				0 & \Id \\
				-R_i & \lambda T_i - Q_i
			\end{pmatrix}
		\]
		and the convergence is uniform on every compact subset of $\RR$. By \eqref{eq:9} it implies that
		if for some $\varepsilon(\lambda) \in \{ -1, 1 \}$
		\[
			\varepsilon(\lambda) \sym{\calF^0(\lambda)} > 0,
		\]
		then for every $j \in \{ 0, 1, \ldots, N-1 \}$
		\[
			\varepsilon(\lambda) \sym{\calF^j(\lambda)} > 0.
		\]
		The proof is complete.
	\end{proof}

\subsection{Asymptotics of generalised eigenvectors}
This section is devoted to show the implications of the non-degeneracy of $(Q^z : z \in \Lambda)$ 
together with the positivity of $|S_n|$ to the asymptotics of the generalised eigenvectors. 

\begin{theorem}
	\label{thm:1}
	Let the family $\{Q^z : z \in K\}$ defined in \eqref{eq:16} be uniformly 
	non-degenerated on a~compact set $K$. Suppose that there are
	$c \geq 1$ and $M' > 0$ such that for all $\alpha \in \calH \oplus \calH$ such that 
	$\norm{\alpha} = 1$, $z \in K$ and $n \geq M$ 
	\begin{equation}
		\label{eq:17}
		c^{-1} \leq \abs{S_n(\alpha, z)} \leq c.
	\end{equation}
	Then there is $c \geq 1$ such that for all $z \in K$, $n \geq 1$ and for every generalised
	eigenvector $u$ corresponding to $z$
	\[
		c^{-1} (\norm{u_0}^2 + \norm{u_1}^2) 
		\leq \norm{a_{n+N-1}} (\norm{u_{n-1}}^2 + \norm{u_n}^2) 
		\leq c (\norm{u_0}^2 + \norm{u_1}^2).
	\]
\end{theorem}
\begin{proof}
	Let $z \in K$ and let $u$ be a~generalised eigenvector corresponding to $z$ such that
	$(u_0, u_1)^t = \alpha$, $\norm{\alpha} = 1$. Since $\{Q^z : z \in K \}$ is uniformly 
	non-degenerated, there are $c \geq 1$ and $M \geq M'$ such that for all $n \geq M$
	\[
		c^{-1} \norm{a_{n+N-1}} (\norm{u_{n-1}}^2 + \norm{u_n}^2) 
		\leq 
		|S_n(\alpha, z)|
		\leq c \norm{a_{n+N-1}} (\norm{u_{n-1}}^2 + \norm{u_n}^2) ,
	\]
	which together with \eqref{eq:17} implies that there is $c \geq 1$ such that for all $n \geq M$
   \[
		c^{-1} \leq \norm{a_{n+N-1}} (\norm{u_{n-1}}^2 + \norm{u_n}^2)  \leq c.
  	\]
  	
  	For the general non-zero $\alpha$ we use the fact that
  	\[
  		S_n \left( \frac{\alpha}{\norm{\alpha}}, z \right) = 
  		\frac{1}{\norm{\alpha}^2} S_n(\alpha, z)
  	\]
  	and generalised eigenvectors depend linearly on the initial conditions.
\end{proof}
   
\begin{corollary} \label{cor:2}
	Suppose that the assumptions of Theorem~\ref{thm:1} are satisfied.
	Let $\Omega \subset \calH \oplus \calH \setminus \{ 0 \}$ be a~bounded closed set and
	let $K \subset \Lambda$ be a~compact set.
	Assume that for $N$-periodic sequence of self-adjoint operators $(D_n : n \geq 0)$
	\begin{equation} \label{eq:23}
		\lim_{n \rightarrow \infty} 
		\left\|
			\frac{1}{\norm{a_{n+N-1}}}
			\sym{
				\begin{pmatrix}
					a_{n+N-1} & 0 \\
					0 & a_{n+N-1}^*
				\end{pmatrix}
				E X_n(z)
			}
			-
			\begin{pmatrix}
				D_n & 0 \\
				0 & D_n
			\end{pmatrix}
		\right\|
		=
		0
	\end{equation}
	uniformly on $K$ and
	\[
		g(\alpha, z) 
		= 
		\lim_{n \rightarrow \infty} S_n(\alpha, z)
	\]
	uniformly on $\Omega \times K$.
	Then
	\[
		\lim_{n \rightarrow \infty}
			\norm{a_{n+N-1}} ( \sprod{D_n u_{n-1}}{u_{n-1}}_\calH + \sprod{D_n u_{n}}{u_{n}}_\calH) = g
	\]
	uniformly on $\Omega \times K$.
\end{corollary}
\begin{proof}
	Fix $\varepsilon > 0$. By \eqref{eq:23} there is $M$ such that for all $n \geq M$, $z \in K$ 
	and $v \in \calH \oplus \calH$
	\[
		|Q_n^z(v) - ( \sprod{D_n v_1}{v_1}_\calH + \sprod{D_n v_2}{v_2}_\calH) |
		\leq \varepsilon \norm{v}^2.
	\]
	Hence,
	\[
		|S_n -  \norm{a_{n+N-1}}( \sprod{D_n u_{n-1}}{u_{n-1}}_\calH + \sprod{D_n u_{n}}{u_{n}}_\calH) |
		\leq \varepsilon \norm{a_{n+N-1}} (\norm{u_{n-1}}^2 + \norm{u_{n}}^2)
	\]
	uniformly on $\Omega \times K$.
	By Theorem~\ref{thm:1} there is a~constant $c' > 0$ such that 
	\[
		|S_n -  \norm{a_{n+N-1}}( \sprod{D_n u_{n-1}}{u_{n-1}}_\calH + \sprod{D_n u_{n}}{u_{n}}_\calH) |
		\leq \varepsilon c'
	\]
	uniformly on $\Omega \times K$. The proof is complete.
\end{proof}

\subsection{The proof of the convergence}
In this section we are going to prove that the sequence $(S_n : n \geq 0)$ is convergent, which
leads to the proofs of Theorem \ref{thm:B} and \ref{thm:C}.

Let us begin with the main algebraic part of the proof.
	\begin{lemma} \label{lem:3}
		Let $u$ be a~generalised eigenvector associated with $z \in \CC$ and 
		$\alpha \in \calH \oplus \calH$. Then
		\begin{multline*}
			\frac{|S_{n+1}(\alpha, z) - S_n(\alpha, z)|}{\norm{u_{n-1}}^2 + \norm{u_n}^2} \leq 
			\norm{ X_n(z)} \norm{a_{n+N}} \big( \norm{a_{n+N}^{-1} a_{n+N-1}^* 
			- a_n^{-1} a_{n-1}^*} + \\
			|z| \norm{a_{n+N}^{-1} - a_{n}^{-1}} + |z - \overline{z}| \norm{a_{n+N}^{-1}} 
			+ \norm{a_{n+N}^{-1} b_{n+N} - a_{n}^{-1} b_n} \big).
		\end{multline*}
	\end{lemma}
	\begin{proof}
		The formula \eqref{eq:1} implies
		\begin{align*}
	    	S_{n+1}(\alpha, z) & =
	        \bigg\langle 
	        \sym{
		        \begin{pmatrix}
		        	a_{n+N} & 0 \\
		        	0 & a_{n+N}^*
		        \end{pmatrix}
		        E X_{n+1}(z)
	        }
	        \begin{pmatrix}
	        u_{n}\\
	        u_{n+1}
	        \end{pmatrix}
	        ,
	        \begin{pmatrix}
	        u_n\\
	        u_{n+1}
	        \end{pmatrix}
	        \bigg\rangle\\
	        & =
	        \bigg\langle 
	        (B_n(z))^*
	        \sym{
		        \begin{pmatrix}
		        	a_{n+N} & 0 \\
		        	0 & a_{n+N}^*
		        \end{pmatrix}
		        E X_{n+1}(z)
	        }
	        B_n(z)
	        \begin{pmatrix}
	        u_{n-1}\\
	        u_{n}
	        \end{pmatrix}
	        ,
	        \begin{pmatrix}
	        u_{n-1}\\
	        u_{n}
			\end{pmatrix}
			\bigg\rangle.
		\end{align*}
	
		Therefore, by the formulas \eqref{eq:9} and \eqref{eq:11}
		\begin{equation} \label{eq:4}
			S_{n+1} - S_n =
			\bigg\langle
				\sym{C_n(z)}
			  \begin{pmatrix}
	        u_{n-1}\\
	        u_{n}
	        \end{pmatrix}
	        ,
	        \begin{pmatrix}
	        u_{n-1}\\
	        u_{n}
			\end{pmatrix}
			\bigg\rangle,
		\end{equation}
		where
		\[
			C_n(z) =
			(B_n(z))^* 
			\begin{pmatrix}
				a_{n+N} & 0 \\
				0 & a_{n+N}^*
			\end{pmatrix}
			E X_{n+1}(z) B_n(z)
			-
			\begin{pmatrix}
				a_{n+N-1} & 0 \\
				0 & a_{n+N-1}^*
			\end{pmatrix}
			E X_n(z).
		\]		
		
		By using $E^{-1} = -E$, we can write
	   	\[
			(B_n(z))^* 
			\begin{pmatrix}
				a_{n+N} & 0 \\
				0 & a_{n+N}^*
			\end{pmatrix}
			E X_{n+1}(z) B_n(z)
			=
			- (B_n(z))^* 
			\begin{pmatrix}
				a_{n+N} & 0 \\
				0 & a_{n+N}^*
			\end{pmatrix}
			E B_{n+N}(z) E E X_{n}(z).
	   	\]
		Hence,
		\[
			C_n(z) = 
			-\left[ 
			(B_n(z))^* 
			\begin{pmatrix}
				a_{n+N} & 0 \\
				0 & a_{n+N}^*
			\end{pmatrix}
			E B_{n+N}(z) E
			+
			\begin{pmatrix}
				a_{n+N-1} & 0 \\
				0 & a_{n+N-1}^*
			\end{pmatrix}
			\right]
			E X_n(z).
		\]
		Now we can compute
	    \begin{multline*}
			(B_n(z))^* 
			\begin{pmatrix}
				a_{n+N} & 0 \\
				0 & a_{n+N}^*
			\end{pmatrix}
			E B_{n+N}(z) E \\ 
			= 
	        \begin{pmatrix}
		        0 & -a_{n-1} (a_n^*)^{-1} \\
	    	    \Id & (\overline{z} \Id - b_n) (a_n^*)^{-1}
	        \end{pmatrix}
	        \begin{pmatrix}
	        	0 & -a_{n+N} \\
	        	a_{n+N}^* & 0
	        \end{pmatrix}
	        \begin{pmatrix}
		        \Id & 0 \\
	    	    a_{n+N}^{-1} (\lambda \Id - b_{n+N}) & a_{n+N}^{-1} a_{n+N-1}^*
	        \end{pmatrix} \\
			=
	        \begin{pmatrix}
	        	-a_{n-1} (a_n^*)^{-1} a_{n+N}^* & 0 \\
		        (\overline{z} \Id - b_n) (a_n^*)^{-1} a_{n+N}^* & -a_{n+N}
	        \end{pmatrix}
	        \begin{pmatrix}
	    	    \Id & 0 \\
	        	a_{n+N}^{-1}(z \Id - b_{n+N}) & a_{n+N}^{-1} a_{n+N-1}^*
	        \end{pmatrix} \\
			=
	        \begin{pmatrix}
	        -a_{n-1} (a_n^*)^{-1} a_{n+N}^* & 0 \\
	        (\overline{z} \Id - b_n)(a_n^*)^{-1} a_{n+N}^* -(z \Id - b_{n+N}) & -a_{n+N-1}^*
	        \end{pmatrix}.
		\end{multline*}
	    Therefore,
	    \[
			C_n(z) =-
	        \begin{pmatrix}
	        	-a_{n-1}(a_n^*)^{-1} a_{n+N}^* + a_{n+N-1} & 0 \\
	        	(\overline{z} \Id - b_n)(a_n^*)^{-1} a_{n+N}^* -(z \Id - b_{n+N}) & 0
	        \end{pmatrix}
	        E X_n(z).
		\]
	   In particular we can estimate
		\begin{multline*}
			\norm{C_n(z)}
	        \leq
	        \norm{X_n(z)} \norm{a_{n+N}^*} \big( 
	        \norm{a_{n+N-1} (a_{n+N}^*)^{-1} - a_{n-1}(a_{n}^*)^{-1}} \\ 
	        + 
	        |z| \norm{(a_n^*)^{-1} - (a_{n+N}^*)^{-1}} + 
	        |z - \overline{z}| \norm{a_{n+N}^{-1}} +
	        \norm{b_{n+N}(a_{n+N}^*)^{-1} - b_n (a_n^*)^{-1}} \big).
		\end{multline*}
		Therefore, by the last inequality together with \eqref{eq:4}, Schwarz inequality and \eqref{eq:12}
		the result follows.
	\end{proof}

The main result of this section is the following theorem.
\begin{theorem}
	\label{thm:2}
	Assume that for an integer $N \geq 1$
	\begin{enumerate}[(a)]
		\item 
		$
		\begin{aligned}[t]
			\calV_N \bigg( a_n^{-1} : n \geq 0 \bigg) 
			+ \calV_N \bigg( a_{n}^{-1} b_n : n \geq 0 \bigg)
			+ \calV_N \bigg(a_{n}^{-1} a_{n-1}^* : n \geq 1 \bigg) < \infty;
		\end{aligned}
		$ \label{thm:2:eq:1}
		\item $
		\begin{aligned}[t]
		\frac{\norm{a_{n+1}}}{\norm{a_n}} < c_1
		\end{aligned}
		$
		for a~constant $c_1 > 0$ and all $n \in \NN$; \label{thm:2:eq:2}
		\item the family defined in \eqref{eq:16}
		$
		\begin{aligned}[t]
			\big\{Q^z : z \in K \big\}
		\end{aligned}
		$ \label{thm:2:eq:3}
		is uniformly non-degenerated on a~compact connected set $K$.
	\end{enumerate}
	Then there is $c \geq 1$ such that for every $n \geq 1$, for all $z \in K \cap \RR$ and for every 
	generalised eigenvector $u$ corresponding to $z$ we have
	\[
		c^{-1} (\norm{u_0}^2 + \norm{u_1}^2) 
		\leq \norm{a_n}(\norm{u_{n-1}}^2 + \norm{u_n}^2) 
		\leq c (\norm{u_0}^2 + \norm{u_1}^2).
	\]
	Moreover, if 
	\begin{equation} \label{thm:2:eq:4}
		\sum_{n=0}^\infty \norm{a_{n}^{-1}} < \infty,
	\end{equation}
	then the same conclusion holds for $z \in K$.
\end{theorem}
\begin{proof}
	Let $\Omega \subset \calH \oplus \calH \setminus \{ 0 \}$ be a~connected bounded closed set. 
	Let $S_n$ be a~sequence of functions defined by \eqref{eq:7}. In view of
	Theorem \ref{thm:1}, it is enough to show that there are $c \geq 1$ and $M > 0$ such that
	\begin{equation}
		\label{eq:22}
		c^{-1} \leq \abs{S_n(\alpha, z)} \leq c
	\end{equation}
	for all $\alpha \in \Omega$, $z \in K$ and $n > M$. The study of the sequence $(S_n : n \in \NN)$ 
	is motivated by the method developed in \cite{Swiderski2017}.

	Given a~generalised eigenvector corresponding to $z \in K$ such that $(u_0, u_1)^t = \alpha \in \Omega$,
	we can easily see that for each $n \geq 2$ $u_n$, considered as a~function of $\alpha$ and $z$,
	is continuous on $\Omega \times K$. As a~consequence, the function $S_n$ is continuous on
	$\Omega \times K$. Since $\{Q^z : z \in K\}$ is uniformly non-degenerated, there is $M > 0$ 
	such that for each $n \geq M$ the function $S_n$ has no zeros and has the same sign for all $z \in K$
	and $\alpha \in \Omega$. Otherwise, by the connectedness of $\Omega \times K$, there would be 
	$\alpha \in \Omega$ and $z \in K$ such that $S_n(\alpha, z) = 0$, which would contradict 
	the non-degeneracy of $Q_n^z$. 
	
	Next, we define a~sequence of functions $(F_n : n \geq M)$ on $\Omega \times K$ by setting
   	\[
		F_n = \frac{S_{n+1} - S_n}{S_n}.
   	\]
	Then
   	\begin{equation} \label{eq:20}
		\frac{S_n}{S_M} = \prod_{j=M}^{n-1} (1 + F_j).
   	\end{equation}
   
	First of all, let us show that
   	\begin{equation} \label{eq:21}
   		C^{-1} \leq |S_M(\alpha, z)| \leq C
   	\end{equation}
   	for a~constant $C>1$ independent of $\alpha$ and $z$. If it is the case, then by \eqref{eq:20} and
	the fact that each function $F_n$ is continuous, to conclude \eqref{eq:22} it is enough to show
	that the product
	\[
		\prod_{j = M}^n (1 + F_j)
	\]
	converges uniformly on $\Omega \times K$ to a~limit that is bounded away from $0$, which will be
	satisfied if we prove that
	\begin{equation}
		\label{part2:eq:11}
		\sum_{j = M}^\infty \sup_{\alpha \in \Omega} \sup_{z \in K} \abs{F_n(\alpha, z)} < \infty.
	\end{equation}   
   	Let us observe that by \eqref{eq:7} and \eqref{eq:12}
	\begin{equation} \label{eq:18}
		|S_M(\alpha, z)| \leq \norm{a_{M+N-1}} \norm{X_M(z)} 
		(\norm{u_{M-1}(\alpha, z)}^2 + \norm{u_{M}(\alpha, z)}^2).
	\end{equation}
	Moreover, by \eqref{eq:1}
	\begin{equation} \label{eq:31}
		\norm{u_{M-1}(\alpha, z)}^2 + \norm{u_{M}(\alpha, z)}^2 = 
		\langle Y(z) \alpha, Y(z) \alpha \rangle =
		\langle [Y(z)]^* Y(z) \alpha, \alpha \rangle,
	\end{equation}
	for
	\begin{equation} \label{eq:33}
		Y(z) = \prod_{i=1}^{M-1} B_i(z).
	\end{equation}
	Hence,
	\begin{equation} \label{eq:19}
		\norm{u_{M-1}(\alpha, z)}^2 + \norm{u_{M}(\alpha, z)}^2 \leq
		\left[ \prod_{i=1}^{M-1} \norm{B_i(z)}^2 \right] \norm{\alpha}^2.
	\end{equation}
	For every $i$ the function $z \mapsto \norm{B_i(z)}$ is continuous on the compact set $K$.
	Hence, it is uniformly bounded. Furthermore, by the boundedness of $\Omega$ one has that $\norm{\alpha}$ 
	is bounded as well. It shows that the right-hand side of \eqref{eq:19} is uniformly bounded on 
	$\Omega \times K$. Similarly,
	\[
		\norm{X_M(z)} \leq \prod_{i=M}^{M+N-1} \norm{B_{i}(z)}
	\]	
	is uniformly bounded. It implies that the right-hand side of \eqref{eq:18} is uniformly bounded as well.
	Thus, the upper bound in the inequality \eqref{eq:21} is proved. To prove the lower bound, let us see
	that the uniform non-degeneracy implies 
	\begin{equation} \label{eq:32}
		|S_M(\alpha, z)| \geq \norm{a_{N+M-1}} (\norm{u_{M-1}(\alpha, z)}^2 + \norm{u_{M}(\alpha, z)}^2)
	\end{equation}
	for a~constant $c>0$ independent of $\alpha$ and $z$. 
	So by \eqref{eq:31} it remains to show that $[Y(z)]^* Y(z)$ is a~strictly positive operator 
	uniformly with respect to $z \in K$. It will be implied by the uniform bound on 
	$\norm{([Y(z)]^* Y(z))^{-1}}$.
	According to \eqref{eq:33}
	\[
		\norm{([Y(z)]^* Y(z))^{-1}} \leq \prod_{i=1}^{M-1} \norm{B_i^{-1}(z)}^2
	\]
	and by \eqref{eq:2}, as in \eqref{eq:19}, the right-hand side of this inequality is uniformly bounded on 
	$K$. Hence, by \eqref{eq:31} there is a~constant $c'>0$ such that 
	\[	
		\norm{u_{M-1}(\alpha, z)}^2 + \norm{u_{M}(\alpha, z)}^2 \geq c' \norm{\alpha}^2.
	\]
	Consequently, by the positive distance of $\Omega$ to $0$ and \eqref{eq:32}, we proved the remaining 
	lower bound in \eqref{eq:21}.
	
	It remains to prove \eqref{part2:eq:11}. Let $u$ be a~generalised eigenvector corresponding to 
	$z \in K$ such that $(u_0, u_1)^t = \alpha \in \Omega$. 
	In view of \eqref{thm:2:eq:1}, each subsequence $(B_{kN+j}(z) : k \in \NN)$ is uniformly convergent,
	and consequently, the norms $\| X_n(z) \|$ are uniformly bounded with respect to $n$ and 
	$z \in K$.
	Moreover, since $\{Q(z) : z \in K\}$ is uniformly non-degenerated
	\[
		\abs{S_n(\alpha, z)} \geq c^{-1} \norm{a_{n+N-1}} (\norm{u_{n-1}}^2 + \norm{u_n}^2)
	\]
	for $n \geq M$. Therefore, by Lemma~\ref{lem:3}
	\begin{multline} \label{part2:eq:szybkoscZbieznosciReg}
		\abs{F_n(\alpha, z)} \leq
		c c' c_1 \big( \norm{a_{n+N}^{-1} a_{n+N-1}^* - a_n^{-1} a_{n-1}^*} 
		+ |z| \norm{a_{n+N}^{-1} - a_{n}^{-1}} \\
		+ |z - \overline{z}| \norm{a_{n+N}^{-1}} 
		+ \norm{a_{n+N}^{-1} b_{n+N} - a_{n}^{-1} b_n} \big)
	\end{multline}
	for every $\alpha \in \Omega$.
	Using \eqref{thm:2:eq:2}, we can estimate 
	\[
		\begin{aligned}
		\sum_{n = M}^\infty \sup_{\alpha \in \Omega} \sup_{z \in K} \abs{F_n(\alpha, z)} 
		& \leq
		c c' c_1 \calV_N ( a_{n}^{-1} a_{n-1}^{*} : n \geq M )
		+
		c c' c_1 \calV_N ( a_{n}^{-1} b_n : n \geq M) \\
		& \phantom{\leq}+
		c c' c_1 \sup_{z \in K} |z| \calV_N ( a_n^{-1} : n \geq M) 		
		+ c c' c_1 \sup_{z \in K} |z - \overline{z}| \sum_{n=M}^\infty \norm{a_{n}^{-1}}.
		\end{aligned}
	\]
	Thus, \eqref{thm:2:eq:1} and \eqref{thm:2:eq:4} implies \eqref{eq:22}.
	If condition \eqref{thm:2:eq:4} is not satisfied consider $K \cap \RR$ instead $K$ in the last inequality.
	The proof is complete.
\end{proof}

The following Corollary provides an estimate, which in the scalar case expresses the bound on the rate
of the convergence of Turán determinants to the density of the spectral measure of $A$ 
(see \cite{Swiderski2017b}). It follows from the standard proof of the convergence of infinite products 
of numbers.
\begin{corollary} \label{cor:3}
	Under the hypothesis of Theorem \ref{thm:2}, for every bounded and closed 
	$\Omega \subset \calH \oplus \calH \setminus \{ 0 \}$ the sequence of continuous functions 
	$(S_n : n \in \NN)$ converges uniformly on
	$\Omega \times (K \cap \RR)$ (or on $\Omega \times K$ if \eqref{thm:2:eq:4} is satisfied) to the function
	$g$ bounded away from $0$. Moreover, by \eqref{part2:eq:szybkoscZbieznosciReg} there is a~constant 
	$c>0$ such that for all $m > 0$
	\begin{multline*}
		\sup_{\alpha \in \Omega} \sup_{z \in K \cap \RR} |g(\alpha, \lambda) - S_m(\alpha, z)| 
		\leq c \calV_N ( a_{n}^{-1} a_{n-1}^{*} : n \geq m )
		+
		c \calV_N ( a_n^{-1} : n \geq m) \\		
		+
		c \calV_N ( a_{n}^{-1} b_n : n \geq m).
	\end{multline*}
\end{corollary}

Finally, we are ready to prove Theorems \ref{thm:B} and \ref{thm:C}.
\begin{proof}[Proof of Theorem~\ref{thm:B}]
	By Propositions \ref{prop:6} and \ref{prop:8} we have that the assumptions of Theorem~\ref{thm:2}
	are satisfied. Therefore, the result follows from Theorem~\ref{thm:1}.
\end{proof}

\begin{proof}[Proof of Theorem~\ref{thm:C}]
	Since every $C_n$ is invertible, we have
	\[
		\lim_{n \rightarrow \infty} \left\| \left( \frac{a_n}{\norm{a_n}} \right)^{-1} - C_n^{-1} \right\| = 0.
	\]
	Hence, for some $c > 0$
	\[
		\norm{a_n} \norm{a_n^{-1}} \leq c.
	\]
	Consequently,
	\[
		\norm{a_n^{-1}} \leq \frac{c}{\norm{a_n}}
	\]
	and \eqref{thm:2:eq:4} is satisfied. 
	Moreover, it implies that $T_n \equiv 0$ so, in the notation of Proposition~\ref{prop:5}, every 
	$\calF^i(\cdot)$ is constant. Hence, Proposition~\ref{prop:8} implies the almost uniform non-degeneracy of 
	$\{ Q^z : z \in \RR \}$. Since $\calF^i(\cdot)$ is constant on $\CC$ Proposition~\ref{prop:8} implies
	that $\{ Q^z : z \in \CC \}$ is almost uniformly non-degenerated as well. Thus, the assumptions of
	Theorem~\ref{thm:2} are satisfied, and consequently, Theorem~\ref{thm:1} implies the requested asymptotics.
	Finally, Corollary~\ref{cor:4} finishes the proof.
\end{proof}

\section{Exact asymptotics of generalised eigenvectors} \label{sec:exactAsym}
The following Theorem is a~vector valued version of \cite[Corollary 1]{Swiderski2017b}.
\begin{theorem} \label{thm:asym}
	Let $\Omega \subset \calH \oplus \calH \setminus \{ 0 \}$ be a~bounded and closed set and
	let $K \subset \RR$ (or $K \subset \CC$ whether the Carleman condition is not satisfied) be a~compact set. 
	Let $N$ be an odd integer. Let the hypotheses of Theorem~\ref{thm:B} be satisfied. Assume further
	that 
	\[
		T_n \equiv 0, \quad Q_n \equiv 0, \quad R_n \equiv \Id, \quad C_n \equiv C.
	\]
	Then $C = C^*$ and
	\[
		\lim_{n \rightarrow \infty} \norm{a_{n}}(\sprod{C u_{n-1}}{u_{n-1}}_\calH 
		+ \sprod{C u_{n}}{u_{n}}_\calH) = g
	\]
	uniformly on $\Omega \times K$, where 
	\[
		g(\alpha, z) = \lim_{n \rightarrow \infty} S_n(\alpha, z),
	\]
	for $S_n$ defined in \eqref{eq:7}.
\end{theorem}
\begin{proof}
	We have
	\[
		\begin{pmatrix}
			0 & \Id \\
			-\Id & 0
		\end{pmatrix}^2
		= 
		-\begin{pmatrix}
			\Id & 0 \\
			0 & \Id
		\end{pmatrix}.
	\]
	Hence,
	\[
		\begin{pmatrix}
			0 & -C \\
			C^* & 0
		\end{pmatrix}
		\begin{pmatrix}
			0 & \Id \\
			-\Id & 0
		\end{pmatrix}^N
		=
		(-1)^{(N-1)/2}
		\begin{pmatrix}
			C & 0 \\
			0 & C^*
		\end{pmatrix}.
	\]
	Consequently,
	\[
		\calF(\lambda)
		=
		\begin{pmatrix}
			\sym{C} & 0 \\
			0 & \sym{C}
		\end{pmatrix}.
	\]
	Therefore, by Proposition~\ref{prop:6} $r \Id = C^{-1} C^*$ for $r = \norm{C^{-1} C^*}$. It implies
	that $r C = C^*$. Taking norms we obtain $r=1$, and consequently, $C = C^*$.
	Moreover, by Corollary~\ref{cor:3}, $g$ is a~continuous function on $\Omega \times K$ which is bounded away 
	from $0$. Hence, by Corollary~\ref{cor:2} the result follows.
\end{proof}

In the scalar case, and under stronger assumptions, the similar results were obtained in \cite{Ignjatovic2014}.
To obtain the complete information of the asymptotics it is of interest to identify the function $g$.
In the scalar case $g$ is related to the density of the spectral measure of $A$ 
(see \cite[Corollary 1]{Swiderski2017b}).

The following Corollary is an extension of \cite[Corollary 3]{Swiderski2017b} to the operator case.
In the scalar case it provides exact asymptotics of the so-called Christoffel functions,
which have applications, e.g. in random matrix theory (see \cite{Lubinsky2016}) 
or signal processing (see \cite{Ignjatovic2009}).
We believe that in the operator case it will also have some applications.
\begin{corollary}
	Let the assumptions of Theorem~\ref{thm:asym} be satisfied. Assume further that
	\[
		\sum_{k=0}^\infty \frac{1}{\norm{a_k}} = \infty.
	\]
	Then
	\[
		\lim_{n \rightarrow \infty} 
		\left[ \sum_{k=0}^n \frac{1}{\norm{a_k}}\right]^{-1} 
		\sum_{k=0}^n \sprod{C u_{k}}{u_{k}}_\calH
		= \frac{1}{2} g
	\]
	uniformly on $\Omega \times K$, where 
	\[
		g(\alpha, z) = \lim_{n \rightarrow \infty} S_n(\alpha, z),
	\]
	for $S_n$ defined in \eqref{eq:7}.
\end{corollary}
\begin{proof}
	By Stolz–Cesàro theorem (also known as L'Hôpital's rule for sequences)
	\begin{multline*}
		\lim_{n \rightarrow \infty} 
		\left[ \sum_{k=0}^n \frac{1}{\norm{a_k}}\right]^{-1} 
		\sum_{k=0}^n \sprod{C u_{k}}{u_{k}}_\calH
		=
		\lim_{n \rightarrow \infty}		
		\frac{\sprod{C u_{n-1}}{u_{n-1}}_\calH + \sprod{C u_{n}}{u_{n}}_\calH}
		{1/\norm{a_{n-1}} + 1/\norm{a_n}}
		 \\
		=
		\lim_{n \rightarrow \infty}		
		\frac{\norm{a_{n}}(\sprod{C u_{n-1}}{u_{n-1}}_\calH + \sprod{C u_{n}}{u_{n}}_\calH)}
		{\norm{a_{n}}/\norm{a_{n-1}} + 1}.
	\end{multline*}
	Theorem~\ref{thm:asym} implies that $C=C^*$, and consequently, Proposition~\ref{prop:6} shows that
	$\norm{a_n}/\norm{a_{n-1}}$ tends to $1$.
	Therefore, by Theorem~\ref{thm:asym} the result follows.
\end{proof}

\section{Examples} \label{sec:examples}
\subsection{Examples to Theorem~\ref{thm:A}}
In this section we show examples to the special cases of Theorem~\ref{thm:A} presented in 
Section~\ref{sec:specialCases}, i.e. to Theorems \ref{thm:spec:2} and \ref{thm:spec:3}. 
Since Theorem \ref{thm:spec:1} is a~weaker version of Theorem~\ref{thm:B}, the examples to it are 
postponed to the next section.

\begin{example}
	Assume that $X$ and $Y$ are bounded non-commuting operators on $\calH$ such that $X$ 
	is invertible normal and $Y$ is self-adjoint. Let
	\[
		\tilde{x}_k = k \sqrt{\log(k+1)}, \quad \tilde{y}_k = \frac{1}{k \log(k+1)}.
	\]
	Denote
	\[
		\tilde{x}^k = (\tilde{x}_k : 1 \leq j \leq k), \quad \tilde{y}^k = (\tilde{y}_k : 1 \leq j \leq k),
	\]
	i.e. the $k$th repetition of $\tilde{x}_k$ and $\tilde{y}_k$. We define in the block form
	\[
		x = (\tilde{x}^k : k \geq 1), \quad y = (\tilde{y}^k : k \geq 1).
	\]
	Then for
	\[
		a_n = x_n X, \quad b_n = y_n Y
	\]
	the assumptions of Theorem~\ref{thm:spec:2} are satisfied.
\end{example}
\begin{proof}
	We have
	\[
		a_{n+1} a_{n+1}^* - a_n^* a_n = x_{n+1}^2 X X^* - x_n^2 X^* X
	\]
	which by the monotonicity of $x_n$ and normality of $X$ is positive. Hence, 
	the hypothesis~\eqref{thm:spec:2:eq:1} is satisfied.
	
	Next, one has $\norm{a_n} = x_n \norm{X}$. Therefore, by
	\[
		\sum_{n=0}^\infty \frac{1}{x_n^2} = 
		\sum_{k=1}^\infty \frac{k}{\tilde{x}_k^2} =
		\sum_{k=1}^\infty \frac{1}{k \log(k+1)} = \infty
	\]
	we obtain the hypothesis~\eqref{thm:spec:2:eq:3}.
	
	Finally,
	\[
		\frac{\norm{a_n b_{n+1} - b_n a_n}}{x_n^2} \leq \frac{|y_{n+1} - y_n|}{x_n} \norm{X Y}
		+
		\frac{|y_n|}{x_n} \norm{X Y - Y X}
	\]
	and by the fact that $(x_{n+1}/x_n : n \geq 0)$ tends to $1$, the hypothesis~\eqref{thm:spec:2:eq:2} 
	is will be satisfied if $(y_n/ x_n : n \geq 0)$ is summable. But
	\[
		\sum_{n=0}^\infty \frac{y_n}{x_n} = 
		\sum_{k=1}^\infty k \frac{\tilde{y}_k}{\tilde{x}_k} =
		\sum_{k=1}^\infty \frac{1}{k [\log(k+1)]^{3/2}} < \infty
	\]
	and the result follows.
\end{proof}

\begin{example}
	Let $K \geq 1$ be an integer and $M$ be such that $\log^{(K)}(M) > 0$ (see \eqref{eq:30}).
	Assume that $X$ and $Y$ are bounded non-commuting self-adjoint operators on $\calH$ such that 
	$X$ is invertible. Let
	\[
		a_n = x_n X, \quad b_n = y_n Y,
	\]
	for
	\[
		x_n = (n+M) g_K(n+M), \quad y_n = \frac{1}{\log^{(K)}(n+M)}.
	\]
	Then the assumptions of Theorem~\ref{thm:spec:3} are satisfied.
\end{example}
\begin{proof}
	The hypotheses \eqref{thm:spec:3:eq:1} and \eqref{thm:spec:3:eq:4} from Theorem~\ref{thm:spec:3}
	are straightforward.
	
	Since $X$ is self-adjoint
	\[
		(a_{n-1}^*)^{-1} a_{n} = \frac{x_{n}}{x_{n-1}} \Id.
	\]
	Therefore, by \cite[Example 4.5]{Swiderski2016} the hypothesis \eqref{thm:spec:3:eq:2} 
	is satisfied.
	
	It remains to show the hypothesis~\eqref{thm:spec:3:eq:3}. We have
	\[
		a_n^{-1} b_n - b_{n+1} a_n^{-1} = 
		\frac{y_n - y_{n+1}}{x_n} X^{-1} Y + 
		\frac{y_{n+1}}{x_n} (X^{-1} Y - Y X^{-1}).
	\]
	Since $(y_{n+1}/x_n : n \geq 0)$ tends to $1$ it remains to show that $(y_n/x_n : n \geq 0)$ is summable.
	But
	\[
		\frac{y_n}{x_n} = \frac{1}{(n+M) g_{K-1}(n+M) [\log^{(K)}(n+M)]^2},
	\]
	which by the Cauchy condensation test applied $K$ times is summable. The proof is complete.
\end{proof}

\subsection{Examples to Theorems 2 and 3}
	The following Proposition provides a~simple way of the construction of sequences satisfying the 
	bounded variation condition of Theorem~\ref{thm:B}.
	\begin{proposition} \label{prop:7}
		Fix $N \geq 1$ and a~Hilbert space $\calH$.
		Let $(x_n : n \geq 0)$ and $(y_n : n \geq 0)$ be sequences of numbers such that $x_n > 0$,
		$b_n \in \RR$ and
		\[
			\calV_N \left( \frac{x_{n-1}}{x_n} : n \geq 1 \right) + 
			\calV_N \left( \frac{y_n}{x_n} : n \geq 0 \right) + 
			\calV_N \left( \frac{1}{x_n} : n \geq 0 \right) < \infty.
		\]
		Let $(X_n : n \in \ZZ)$ and $(Y_n : n \in \ZZ)$ be $N$-periodic sequences of bounded operators on
		$\calH$ such that for every $n$ each $X_n$ is invertible and each $Y_n$ is self-adjoint. 
		Let us define
		\[
			a_n = x_n X_n, \quad b_n = y_n Y_n.
		\]
		Then
		\[
			\calV_N(a_n^{-1} a_{n-1}^* : n \geq 1) +
			\calV_N(a_n^{-1} b_n : n \geq 0) + 
			\calV_N(a_n^{-1} : n \geq 0) < \infty.
		\]
	\end{proposition}
	\begin{proof}
		We have
		\[
			a_n^{-1} a_{n-1}^* = \left( \frac{x_{n-1}}{x_n} \Id \right) \left( X_{n}^{-1} X^*_{n-1} \right), \quad
			a_{n}^{-1} b_n = \left( \frac{y_n}{x_n} \Id \right) \left( X^{-1}_n Y_n \right), \quad 
			a_n^{-1} = \left( \frac{1}{x_n} \Id \right) X_n^{-1}.
		\]
		Therefore, it is enough to apply Proposition~\ref{prelim:normedVariation}.
	\end{proof}
	
	The next Proposition provides a~convenient form of $\calF(\lambda)$ for $N=1$.
	\begin{proposition} \label{prop:2}
		Assume
		\begin{enumerate}[(a)]
			\item $\begin{aligned}
				\lim_{n \rightarrow \infty} \norm{a_n} = a \in (0, \infty]
			\end{aligned}$,		
		
			\item $\begin{aligned}
				\lim_{n \rightarrow \infty} \frac{a_n}{\norm{a_n}} = C
			\end{aligned}$,
			
			\item $\begin{aligned}
				\lim_{n \rightarrow \infty} \frac{b_n}{\norm{a_n}}= D
			\end{aligned}$,
			
			\item $\begin{aligned}
				\lim_{n \rightarrow \infty} \frac{\norm{a_{n-1}}}{\norm{a_n}}= 1
			\end{aligned}$.
		\end{enumerate}
		
		Then, in the notation of Theorem~\ref{thm:B}
		\[
			\calF(\lambda)
			=
			\begin{pmatrix}
				\sym{C} & \frac{1}{2} D - \frac{\lambda}{2a} \Id\\
				\frac{1}{2} D - \frac{\lambda}{2a} \Id & \sym{C}
			\end{pmatrix}.
		\]		
	\end{proposition}
	\begin{proof}
		Since
		\[
			a_n^{-1} b_n = \left( \frac{a_n}{\norm{a_n}} \right)^{-1} \frac{b_n}{\norm{a_n}}, \quad
			a_{n}^{-1} a_{n-1}^* = \left( \frac{a_n}{\norm{a_n}} \right)^{-1} 
			\frac{a_{n-1}^*}{\norm{a_{n-1}}} \frac{\norm{a_{n-1}}}{\norm{a_n}}
		\]
		we have
		\[
			Q_0 = C^{-1} D, \quad R_0 = C^{-1} C^*.
		\]
		Hence, the direct computation shows that $\calF(\lambda)$ has the requested form.
	\end{proof}

	In the following Example we discuss the optimality of $\Lambda$ in the case of constant coefficients.
	\begin{example}
		Let 
         \[
            a_n = 
            \begin{pmatrix}
               1 & 1\\
               1 & 2
            \end{pmatrix}, \quad
            b_n = 
            \begin{pmatrix}
               2 & 1\\
               1 & 1
            \end{pmatrix}.
         \]
		Then the assumptions of Theorem~\ref{thm:B} are satisfied with
		\[
			\Lambda = \left( \frac{-3+\sqrt{13}}{2}, \frac{9-\sqrt{37}}{2} \right) \supset [0.303, 1.458].
		\]
	 	Moreover, $\Lambda$ is the maximal set where $A$ has absolutely continuous spectrum of the 
	 	multiplicity~$2$.
	\end{example}
	\begin{proof}
		Let
		\[
			M_1 = \left( \frac{-3-\sqrt{13}}{2}, \frac{-3+\sqrt{13}}{2} \right) \cup 
			\left( \frac{9-\sqrt{37}}{2}, \frac{9+\sqrt{37}}{2} \right), \quad
			M_2 = \left( \frac{-3+\sqrt{13}}{2}, \frac{9-\sqrt{37}}{2} \right).
		\]
		Since $(a_n : n \geq 0)$ and $(b_n : n \geq 0)$ are constant it is sufficient to show that
		matrix $\calF(\lambda)$ is positive definite for $\lambda \in M_2$.
		
		According to Proposition~\ref{prop:2} we have
		\[
			\norm{a_n} \calF(\lambda) =
			\begin{pmatrix}
				1 & 1 & 1 -\frac{\lambda}{2} & \frac{1}{2} \\
				1 & 2 & \frac{1}{2} & \frac{1}{2} - \frac{\lambda}{2} \\
				1 -\frac{\lambda}{2} & \frac{1}{2} & 1 & 1 \\
				\frac{1}{2} & \frac{1}{2} - \frac{\lambda}{2} & 1 & 2
			\end{pmatrix}.
		\]
		The determinants of its principal minors are equal to
		\[
			1, \quad 1, \quad - \frac{1}{2} \lambda^2 + \frac{3}{2} \lambda -\frac{1}{4}, \quad
			\frac{1}{16} \lambda^4 -\frac{3}{8} \lambda^3 - \frac{17}{16} \lambda^2 
			+ \frac{21}{8} \lambda -\frac{11}{16}.
		\]
		Hence, the matrix $\calF(\lambda)$ is positively definite whether $\lambda \in M_2$. 
		Moreover, the determinant of the last minor is negative only for $\lambda \in M_1$.
		
		According to \cite[Theorem 3]{Zygmunt2001} the matrix $A$ is purely absolutely continuous on
		the closure of the set $M_1 \cup M_2$. Moreover, the spectrum of $A$ is of multiplicity $1$ and $2$ on 
		$M_1$ and $M_2$, respectively.
	\end{proof}
	
	In the next Example we consider the unbounded case for $N=1$.
	\begin{example}
		Let 
         \[
            X = 
            \begin{pmatrix}
               1 & 1\\
               1 & 2
            \end{pmatrix}, \quad
            Y = 
            \begin{pmatrix}
               2 & 1\\
               1 & 1
            \end{pmatrix}.
         \]
		Let us assume that real sequences $(x_n: n \geq 0)$ and $(y_n: n \geq 0)$ such that 
		$x_n > 0$ and $y_n \in \RR$ for every $n$ satisfy
		\[
			\calV_1 \left( \frac{x_{n-1}}{x_n} : n \geq 1 \right) + 
			\calV_1 \left( \frac{y_n}{x_n} : n \geq 0 \right) + 
			\calV_1 \left( \frac{1}{x_n} : n \geq 0 \right) < \infty
		\]
		and
		\[
			\lim_{n \rightarrow \infty} x_n = \infty, \quad 
			\lim_{n \rightarrow \infty} \frac{x_{n-1}}{x_n} = 1, \quad
			\lim_{n \rightarrow \infty} \frac{y_n}{x_n} = q \in (\sqrt{5}-3, 3-\sqrt{5}).
		\]
		For example: $x_n = (n+1)^\alpha, y_n = q a_n$ for $\alpha > 0$.
		
		Then for
		\[
			a_n = x_n X, \quad b_n = y_n Y
		\]
		the assumptions of Theorem~\ref{thm:B} are satisfied.
	\end{example}
	\begin{proof}
		In view of Proposition~\ref{prop:7} it is enough to show that $\calF$ is positive definite.
		In the notation of Proposition~\ref{prop:2}
		\[
			C = \frac{1}{\norm{X}} X, \quad D = \frac{q}{\norm{X}} Y, \quad a = \infty.
		\]
		Hence, by Proposition~\ref{prop:2}
		\[
			\norm{X} \cdot \calF(\lambda) =
			\begin{pmatrix}
				1 & 1 & q & q/2 \\
				1 & 2 & q/2 & q/2 \\
				q & q/2 & 1 & 1 \\
				q/2 & q/2 & 1 & 2
			\end{pmatrix}.
		\]
		The determinants of the principal minors of this matrix are equal to
		\[
			1, \quad 1, \quad -\frac{5}{4} q^2 + 1, \quad \frac{1}{16} q^4 - \frac{7}{4} q^2 + 1.
		\]
		Hence, this matrix is positive definite if and only if
		\[
			q \in (\sqrt{5}-3, 3-\sqrt{5}) \supset [-0.763, 0.763].
		\]
	\end{proof}

\section*{Acknowledgements}
I would like to thank Bartosz Trojan, Ryszard Szwarc and an anonymous referee for some helpful suggestions.

	\bibliographystyle{abbrv} 
	\bibliography{matrix}
\end{document}